\documentclass[a4paper,11pt]{amsproc}

\usepackage{amsmath,amssymb,amsthm,mathtools}
\usepackage[T1]{fontenc}
\usepackage{lmodern}
\usepackage{microtype}
\usepackage{enumitem}
\usepackage{xargs,xcolor}

\definecolor{citegreen}{rgb}{0,0.5,0.15}
\definecolor{linkblue}{rgb}{0,0.4,1}
\definecolor{citebordergreen}{rgb}{0,0.9,0.37}
\usepackage[colorlinks, linkcolor=blue, unicode=true, citecolor=citegreen, linkbordercolor={0 0.4 1}, citebordercolor={0 0.9 0.37}, pagebackref]{hyperref}

\usepackage{amsrefs}

\newtheorem{theorem}{Theorem}[section]
\newtheorem*{theorema}{Theorem A}
\newtheorem*{theoremb}{Theorem B}
\newtheorem{lemma}[theorem]{Lemma}
\newtheorem{proposition}[theorem]{Proposition}
\newtheorem{corollary}[theorem]{Corollary}

\theoremstyle{definition}
\newtheorem{openproblem}[theorem]{Question}
\newtheorem{definition}[theorem]{Definition}
\newtheorem{remark}[theorem]{Remark}

\makeatletter
\newcommand{\namedlabel}[2]{%
   \phantomsection
   \def\@currentlabel{#2}%
   \label{#1}%
}
\makeatother

\newcommand{\bbC}{\mathbb C}
\newcommand{\bbN}{\mathbb N}

\newcommand{\R}{\mathcal R}
\newcommand{\Unitary}{\mathcal U}

\newcommand{\SU}{\operatorname{SU}}
\newcommand{\Tr}{\operatorname{Tr}}

\newcommand{\Ad}{\operatorname{Ad}}

\newcommand{\eps}{\varepsilon}

\title{The hyperfinite \texorpdfstring{II$_1$}{II1}-factor is Ulam stable}

\author{Vadim Alekseev}
\address{Vadim Alekseev, TU Dresden, 01062 Dresden, Germany}
\email{vadim.alekseev@tu-dresden.de}

\author{Andreas Thom}
\address{Andreas Thom, TU Dresden, 01062 Dresden, Germany}
\email{andreas.thom@tu-dresden.de}


\begin{document}

\begin{abstract}
We prove Ulam stability of the hyperfinite \(\mathrm{II}_1\)-factor with
respect to the trace norm on the operator-norm unit ball. More
precisely, every sufficiently additive, multiplicative, unital, \(*\)-preserving
map from the hyperfinite \(\mathrm{II}_1\)-factor into a \(\mathrm{II}_1\)-factor von Neumann
algebra is uniformly close, after passing to a small amplification of the target, to a genuine unital \(*\)-homomorphism. As a
key finite-dimensional ingredient, we establish a dimension-free stability
theorem for matrix algebras in the same trace-norm setting.

As an application, we show that the hyperfinite \(\mathrm{II}_1\)-factor is
isolated among \(\mathrm{II}_1\)-factors with respect to sufficiently accurate
approximate \(*\)-isomorphisms.
\end{abstract}

\maketitle

\tableofcontents

\section{Introduction}

Ulam stability asks, in one form or another, whether an approximately structure preserving map is close to a genuine homomorphism of structures, see \cite[Chapter VI]{Ulam1960}. We intend to study this problem for tracial von Neumann algebras, where the natural metric structure is the  operator-norm unit ball endowed with the normalized trace norm and all the natural operations inherited from the ambient von Neumann algebra, see Definition \ref{def:matrix-approx-star}.

This stability problem sits at the intersection of several existing stability results.  On the group side, the Hilbert--Schmidt stability theorem of Gowers and Hatami for approximate representations of finite groups is a basic precursor for our argument~\cite{Gowers_Hatami}.  However, our perspective is also related in spirit to rigidity phenomena for corona algebras and their connection with Ulam stability, as surveyed in \cite{CoronaRigidity}, in particular to Farah's seminal work, see \cite{Farah-book}, and by later work on Ulam stability  by McKenney and Vignati for classes of $C^*$-algebras~\cite{McKenney-Vignati}.  Moreover, the question is related to metric lifting results for reduced
products~\cite{DeBondtVignatiMetric}, and to recent work of De Bondt and the
second author on automorphism groups of metric reduced products, both in the
symmetric-group setting~\cite{debondtthom} and in the trace-norm
setting~\cite{DeBondtThom2}.

Let $\R$ be the hyperfinite II$_1$-factor.  Roughly speaking, our main result (Theorem \ref{tha} resp.\ Theorem \ref{T.amplified.stability}) states that if $M$ is a II$_1$-factor and
$
   \varphi:\R_{\leq1}\to M_{\leq1}
$
is an $\varepsilon$-unital $L^2$-$\varepsilon$-$*$-homomorphism, meaning that all algebraic relations are preserved on the operator norm unit ball up to an error of order $\varepsilon$ in the $L^2$-norm, then after passing to an amplification $M^t$ of $M$ the map $\varphi$ is uniformly close in $L^2$-norm to a genuine $*$-homomorphism with $t$ arbitrarily close to $1$ as $\varepsilon\to0$.
The main technical result is a stability result for matrix algebras, which is the core of the paper and may be of independent interest; see Theorem \ref{thb} resp.\ Theorem \ref{thm:matrix-stability-reduction}. The proof uses a non-trivial input from representation theory of the unitary group ${\rm U}(n)$ in the form of a uniform gap of the Casimir eigenvalues that allows us to identify the standard representation among all irreducible representations.

As an application of the main theorem, we obtain an isolation phenomenon for
the hyperfinite factor.  More precisely, there exists $\delta_0>0$ such that
if a II$_1$-factor $M$ admits a $L^2$-$\delta_0$-$*$-isomorphism
$
   \R_{\leq1}\to M_{\leq1},
$
then $M\cong\R$; see Theorem
\ref{T.almost-surjective-R-map-forces-hyperfinite}.  Thus $\R$ is isolated,
in a quantitative sense, among II$_1$-factors with respect to $L^2$-approximate
$*$-isomorphisms.

\medskip

There is also an older and closely related Banach- and operator-algebraic
line of work on approximately multiplicative maps and near inclusions.
Johnson proved perturbation results for approximately multiplicative maps
between Banach algebras, with particularly useful consequences for linear
approximately multiplicative maps from amenable Banach algebras, and hence
from nuclear $C^*$-algebras, into dual Banach algebras and von Neumann
algebras~\cite{Johnson1988}.  Near-inclusion and Kadison--Kastler
perturbation results provide another closely related formulation of the same
stability philosophy; see, for example, the work of Phillips--Raeburn,
Christensen, Johnson, Christensen--Sinclair--Smith--White--Winter, and
Hirshberg--Kirchberg--White~\cite{PhillipsRaeburn1979,Christensen1980,
Johnson1994,CSSWW2012,HKW2012}. 

Moreover, there is a parallel group-theoretic tradition, going back at least to
Kazhdan's work on $\varepsilon$-representations of amenable and non-amenable groups~\cite{Kazhdan1982}.
In this setting one asks whether an approximately multiplicative map from a
group into a unitary group is uniformly close to an actual unitary
representation, see for example \cite{BOT2013} and the references therein. In the group setting, the Hilbert--Schmidt stability theorem of Gowers and Hatami for finite groups is a landmark result~\cite{Gowers_Hatami}, see also \cite{dCOT}.

Inspired by the work of Farah \cites{Farah-book}, McKenney and Vignati  \cite{McKenney-Vignati} prove operator-norm Ulam stability for finite-dimensional
$C^*$-algebras and, with von Neumann algebra targets, for AF algebras
obtained by an inductive limit construction~\cite{McKenney-Vignati}, see also the forthcoming paper \cite{AlekseevFarahThom}.  The present paper is
parallel to that theory but differs in two essential respects: the metric is
the normalized trace norm on the operator-norm unit ball, and the conclusion
naturally allows (and likely requires) a small amplification of the target factor.

There are also related local stability phenomena in the trace-norm setting.
Seminal results in this direction include Jung's uniqueness theorem for embeddings
of amenable tracial von Neumann algebras into $\R^\omega$~\cite{Jung} and the work of Hadwin and Shulman on pointwise approximations \cite{HadwinShulmanTracial,HadwinShulmanHS}. The stability problem studied
here is different: we ask for uniform approximation on the entire
operator-norm unit ball.

\medskip

The paper is organized as follows: The core of the paper is a quantitative stability theorem for matrix algebras.  Its proof proceeds in four steps.  First, we show that an $L^2$-$\varepsilon$-$*$-homomorphism approximately preserves the trace, Section \ref{sec:preliminaries}.  Second, we restrict to the unitary group and obtain an approximate unitary representation.  Third, we apply the compact-group stability machinery of de Chiffre, Ozawa and the second author~\cite{dCOT}, see Section \ref{sec:dCOT} and combine it with a representation-theoretic analysis of almost isometric representations of ${\rm U}(n)$, see Section \ref{sec:matrix-stability}.  Once the matrix case is established, the passage to $\R$ uses an increasing sequence of matrix subfactors, a weak*-compactness argument for completely positive maps, and a Stinespring compression argument, see Section \ref{sec:hyperfinite}. Finally, in Section \ref{sec:applications} we record some applications of the main result, including the isolation phenomenon for $\R$ mentioned above. We end with a section on open problems, see Section \ref{sec:open-problems}. 

\section{Preliminaries}\label{sec:preliminaries}

\subsection{Definitions and notation}

In continuous logic, a tracial von Neumann algebra is modelled by its unit ball with respect to the operator norm, equipped with the normalized trace norm. It is thus natural to study stability problems in exactly this setting.

For a tracial von Neumann algebra $(M,\tau)$, we write
\[
M_{\le 1}:=\{x\in M:\|x\|\le 1\},
\qquad
\|x\|_{2,\tau}:=\tau(x^*x)^{1/2}.
\]

\begin{definition}\label{def:matrix-approx-star}
Let $(N,\tau_N)$ and $(M,\tau_M)$ be tracial von Neumann algebras, and let $\varphi\colon N_{\le 1}\to M_{\le 1}$ be a map. We say that $\varphi$ is an \emph{$L^2$-$\varepsilon$-$*$-homomorphism} if for all $x,y\in N_{\le 1}$ with $x+y\in N_{\le 1}$ and all $\lambda\in\bbC$ with $|\lambda|\le 1$ one has
\begin{align*}
\|\varphi(x+y)-\varphi(x)-\varphi(y)\|_{2,\tau_M} &\le \varepsilon,\\
\|\varphi(\lambda x)-\lambda\varphi(x)\|_{2,\tau_M} &\le \varepsilon,\\
\|\varphi(xy)-\varphi(x)\varphi(y)\|_{2,\tau_M} &\le \varepsilon,\\
\|\varphi(x^*)-\varphi(x)^*\|_{2,\tau_M} &\le \varepsilon.
\end{align*}
If in addition
\[
\|\varphi(1)-1\|_{2,\tau_M}\le \varepsilon,
\]
then we call $\varphi$ \emph{$\varepsilon$-unital}.
\end{definition}

A strong Ulam stability result for $(N,\tau_N)$ would say that for every $\varepsilon>0$ there exists $\delta>0$ such that if $\varphi\colon N_{\le 1}\to M_{\le 1}$ is a $\delta$-unital $L^2$-$\delta$-$*$-homomorphism, then there exists a genuine $*$-homomorphism $\psi\colon N\to M$ such that $\varphi$ and $\psi$ are $\varepsilon$-close in the trace norm on $N_{\le 1}$. Such a strong stability result cannot hold for infinite-dimensional von Neumann algebras, see Proposition \ref{P.compression-almost-homomorphism-not-close-to-isomorphism}. However, our main result, Theorem \ref{T.amplified.stability}, is a precise formulation of a weaker stability result, which allows for a small amplification of the target algebra.

\begin{definition} Let $(M,\tau)$ be a tracial von Neumann algebra. We say that $(M,\tau)$ is \emph{Ulam stable} if for every $\varepsilon>0$ there exists $\delta>0$ such that if $(N,\tau_N)$ is a II$_1$-factor von Neumann algebra and $\varphi\colon M_{\le 1}\to N_{\le 1}$ is a $\delta$-unital $L^2$-$\delta$-$*$-homomorphism, then there exist $t\in(1-\varepsilon,1+\varepsilon)$ and a unital normal $*$-homomorphism $\psi\colon M\to N^t$ such that 
$$ \sup_{\|x\|\le 1} \|\varphi(x)-\psi(x)\|_{2,\tau_N^t}\le \varepsilon. $$
\end{definition}

\begin{remark}
We restrict the target algebra in the above definition to II$_1$-factors only
in order to keep the amplification appearing in the conclusion notationally
simple.  This does not amount to a substantial restriction.  Indeed, by a
classical result \cite{ElliottHandelman1981}, every tracial von Neumann
algebra admits a trace-preserving embedding into a II$_1$-factor.
\end{remark}

Our main result, Theorem \ref{T.amplified.stability}, can be stated as follows:

\begin{theorema}\namedlabel{tha}{A}
The hyperfinite {\rm II}$_1$-factor is Ulam stable.
\end{theorema}
Along the way, we will also establish the following stability result for matrix algebras (see Theorem \ref{thm:matrix-stability-reduction} for the precise formulation):
\begin{theoremb}\namedlabel{thb}{B}
Matrix algebras are Ulam stable, uniformly in matrix size.
\end{theoremb}
Theorem B is the core of the paper and the main technical result.

\medskip

Let us first clarify that a stronger form of Ulam stability without amplification fails in general. Let $(M,\tau)$ be a $\mathrm{II}_1$-factor, let $p\in M$ be a non-zero
projection, and put
\[
        t:=\tau(p)>0,
        \qquad
        \tau_p(x):=\frac{\tau(x)}{t}
        \quad (x\in pMp).
\]
Let
$
        \varphi_p:M_{\leq 1}\to (pMp)_{\leq 1},$
        $\varphi_p(x):=pxp
$
be the compression map.

\begin{proposition}
\label{P.compression-almost-homomorphism-not-close-to-isomorphism}
The map $\varphi_p \colon M_{\leq 1} \to (pMp)_{\leq 1}$ is a unital
$L^2$-$\delta$-$*$-homomorphism with respect to the normalized trace
$\tau_p$, where
$
        \delta:=\left(1/t-1\right)^{1/2}.
$
Moreover, if $t<1$ and
$
        \theta \colon M\to pMp
$
is a unital $*$-isomorphism, then
\[
        \sup_{\|x\|\leq 1}
        \|\theta(x)-\varphi_p(x)\|_{2,\tau_p}
        \geq 1 .
\]
\end{proposition}

\begin{proof}
The map $\varphi_p$ takes contractions to contractions. It is homogeneous, additive, and $*$-preserving:
\[
        \varphi_p(\lambda x)=\lambda\varphi_p(x),
        \qquad
        \varphi_p(x+y)=\varphi_p(x)+\varphi_p(y),
        \qquad
        \varphi_p(x^*)=\varphi_p(x)^* .
\]
Moreover,
$
        \varphi_p(1)=p,
$
which is the unit of $pMp$. Thus $\varphi_p$ is unital as a map into
$pMp$.
We now estimate the multiplicative defect. For $x,y\in M_{\leq 1}$,
\[
        \varphi_p(xy)-\varphi_p(x)\varphi_p(y)
        =
        pxyp-pxpyp
        =
        px(1-p)yp .
\]
Therefore
\[
        \|\varphi_p(xy)-\varphi_p(x)\varphi_p(y)\|_{2,\tau_p}
        \leq
        \|px(1-p)\|_{2,\tau_p}.
\]
Since $\|x\|\leq 1$, we get
\[
   \|px(1-p)\|_{2,\tau_p}^2 = \frac{1}{t}\,\tau\bigl((1-p)x^*px(1-p)\bigr) \leq \frac{1}{t}\,\tau(1-p) = \frac{1-t}{t}.
\]
Hence
\[
        \|\varphi_p(xy)-\varphi_p(x)\varphi_p(y)\|_{2,\tau_p}
        \leq
        \left(\frac{1-t}{t}\right)^{1/2}
        =
        \delta .
\]

Assume now that $t<1$ and that
$
        \theta:M\to pMp
$
is a unital $*$-isomorphism. We regard $\theta$ as a
non-unital endomorphism of $M$ whose range lies in $pMp$. Thus
$
        \theta(1)=p .
$
Define projections
\[
        e_0=1,
        \qquad
        e_{n+1}=\theta(e_n)
        \quad(n\geq 0).
\]
Then
$e_0=1$ and $e_1=\theta(1)=p$. Moreover, since $\theta$ is positive,
\[
        e_{n+1}\leq e_n
        \qquad(n\geq 0).
\]
Indeed, this is clear for $n=0$, and if $e_n\leq e_{n-1}$, then
$
        e_{n+1}=\theta(e_n)\leq \theta(e_{n-1})=e_n .
$
Thus
$
        1=e_0\geq e_1\geq e_2\geq\cdots
$
is a decreasing sequence of projections.

Since $\theta$ is trace-preserving from $(M,\tau)$ to $(pMp,\tau_p)$, we have
\[
        \tau_p(\theta(x))=\tau(x)
        \qquad(x\in M).
\]
Equivalently,
$
        \tau(\theta(x))=t\tau(x).
$
Therefore
$\tau(e_n)=t^n$ for all $n\geq 0$.
Set
$d_n=e_n-e_{n+1}$ for all $n\geq 0$.
Then the $d_n$ are pairwise orthogonal projections, and
$
        \theta(d_n)
        =
        \theta(e_n-e_{n+1})
        =
        e_{n+1}-e_{n+2}
        =
        d_{n+1}.
$
For $n\geq 1$, we have $d_n\leq p$, and therefore
\[
        \tau_p(d_n)
        =
        \frac{\tau(d_n)}{t}
        =
        \frac{t^n-t^{n+1}}{t}
        =
        (1-t)t^{n-1}.
\]

For $N\geq 1$, define
\[
        x_N=\sum_{n=1}^N (-1)^n d_n .
\]
Since the $d_n$ are pairwise orthogonal projections, $x_N$ is self-adjoint and
$
        \|x_N\|\leq 1.
$
Also $x_N\in pMp$, because $d_n\leq p$ for all $n\geq 1$. Hence
$
        \varphi_p(x_N)=px_Np=x_N .
$
On the other hand,
\[
        \theta(x_N)
        =
        \sum_{n=1}^N (-1)^n d_{n+1}.
\]
Thus
\[
\begin{aligned}
        \theta(x_N)-\varphi_p(x_N)
        &=
        \theta(x_N)-x_N                                      \\
        &=
        \sum_{n=1}^N (-1)^n d_{n+1}
        -
        \sum_{n=1}^N (-1)^n d_n                               \\
        &=
        d_1
        +
        \sum_{n=2}^N 2(-1)^{n-1}d_n
        +
        (-1)^N d_{N+1}.
\end{aligned}
\]
Since the $d_n$ are pairwise orthogonal projections, we get
\[
\begin{aligned}
        \|\theta(x_N)-\varphi_p(x_N)\|_{2,\tau_p}^2
        &=
        \tau_p(d_1)
        +
        4\sum_{n=2}^N \tau_p(d_n)
        +
        \tau_p(d_{N+1})                                      \\
        &=
        (1-t)
        +
        4\sum_{n=2}^N (1-t)t^{n-1}
        +
        (1-t)t^N .
\end{aligned}
\]
Letting $N\to\infty$, we obtain
\[
\begin{aligned}
        \lim_{N\to\infty}
        \|\theta(x_N)-\varphi_p(x_N)\|_{2,\tau_p}^2
        &=
        (1-t)
        +
        4\sum_{n=2}^\infty (1-t)t^{n-1}              \\
        &=
        (1-t)+4t                                    \\
        &=
        1+3t .
\end{aligned}
\]
Since $\|x_N\|\leq 1$ for all $N$, it follows that
\[
        \sup_{\|x\|\leq 1}
        \|\theta(x)-\varphi_p(x)\|_{2,\tau_p}^2
        \geq 1.
\]
This finishes the proof.
\end{proof}

The first somewhat non-trivial step in the proof of our main result is to show that an $L^2$-$\varepsilon$-$*$-homomorphism from a finite factor approximately preserves the trace. This is the content of the following proposition.

\begin{proposition}
\label{P.uniform-trace-control-finite-factors}
There exists a function
$
\eta:[0,1]\to[0,\infty)$ with $\eta(\varepsilon)\to 0$ as $\varepsilon\to 0$,
with the following property.
Let \((M,\tau_M)\) be a finite factor and \((N,\tau_N)\) be a tracial von Neumann algebra. Let
$
\varphi:M_{\le 1}\to N_{\le 1}
$
be an \(\varepsilon\)-unital $L^2$-\(\varepsilon\)-\(*\)-homomorphism. Then, for all \(x\in M_{\le 1}\),
\[
\bigl|\tau_N(\varphi(x))-\tau_M(x)\bigr|\le \eta(\varepsilon).
\]
\end{proposition}

\begin{proof}
We first record two elementary estimates.
Let \(z_1,\ldots,z_m\in M_{\le 1}\), and suppose that all partial sums
\[
z_1+\cdots+z_k,\qquad 1\le k\le m,
\]
belong to \(M_{\le 1}\). Then repeated use of approximate additivity gives
\[
\left\|
\varphi\left(\sum_{j=1}^m z_j\right)-\sum_{j=1}^m\varphi(z_j)
\right\|_2
\le (m-1)\varepsilon.
\]
If additionally \(|\lambda_j|\le 1\) and all partial sums of
\(\lambda_jz_j\) belong to \(M_{\le 1}\), then approximate homogeneity gives
\[
\left\|
\varphi\left(\sum_{j=1}^m \lambda_jz_j\right)
-\sum_{j=1}^m\lambda_j\varphi(z_j)
\right\|_2
\le (2m-1)\varepsilon.
\tag{1}
\]

Next, equivalent projections have approximately equal image traces. Indeed,
if \(e,f\in M\) are projections and \(e\sim f\), choose a partial isometry
\(v\in M\) with \(v^*v=e\) and \(vv^*=f\). Then
\[
\begin{aligned}
\bigl|\tau_N(\varphi(e))-\tau_N(\varphi(f))\bigr|
&\le
\bigl|\tau_N(\varphi(e)-\varphi(v^*)\varphi(v))\bigr|  \\
&\quad+
\bigl|\tau_N(\varphi(v)\varphi(v^*)-\varphi(f))\bigr|  \\
&\le 2\varepsilon.
\end{aligned}
\tag{2}
\]
Here we used approximate multiplicativity and traciality of \(\tau_N\).

We now prove a uniform trace estimate for projections. Fix an integer
\(L\ge 2\). We distinguish three cases.

\smallskip

\noindent
\emph{Case 1: \(M=M_d(\mathbb C)\) with \(d<L^2\).}
Let \(e\in M_d(\mathbb C)\) be a projection of rank \(k\). If \(k=0\), then
\(\|\varphi(e)\|_2\le\varepsilon\) by approximate homogeneity, and there is
nothing to prove. Otherwise choose mutually orthogonal minimal projections
\(f_1,\ldots,f_d\) with sum \(1\), and put
$
f:=f_1+\cdots+f_k .
$
Then \(e\sim f\). Set
$
a_i:=\tau_N(\varphi(f_i)).
$
By (2), \(|a_i-a_j|\le 2\varepsilon\) for all \(i,j\). Moreover, by (1) and
\(\varepsilon\)-unitality,
\[
\left|\sum_{i=1}^d a_i-1\right|\le d\varepsilon.
\]
Therefore
\[
\left|\sum_{i=1}^k a_i-\frac{k}{d}\right|\le 3d\varepsilon.
\]
Using (1) for \(f=f_1+\cdots+f_k\), and then using \(e\sim f\), we obtain
\[
\bigl|\tau_N(\varphi(e))-\tau_M(e)\bigr|
\le (4d+2)\varepsilon
\le (4L^2+2)\varepsilon .
\tag{3}
\]

\smallskip

\noindent
\emph{Case 2: \(M=M_d(\mathbb C)\) with \(d\ge L^2\).}
We first control projections of small rank. Suppose \(e\) is a projection
with
$
\operatorname{rank}(e)\le \frac dL.
$
Then we can find pairwise orthogonal projections
$
e_1,\ldots,e_L
$
with \(e_1=e\) and \(e_i\sim e\) for all \(i\). Put
$
r:=e_1+\cdots+e_L.
$
By (1),
\[
\left|
\tau_N(\varphi(r))-\sum_{i=1}^L\tau_N(\varphi(e_i))
\right|
\le (L-1)\varepsilon.
\]
Since \(\|\varphi(r)\|\le 1\), we have \(|\tau_N(\varphi(r))|\le 1\). Together
with (2), this implies
\[
|\tau_N(\varphi(e))|
\le \frac1L+3\varepsilon .
\tag{4}
\]

Now put
$
s:=\left\lfloor d/L\right\rfloor,
$ and $
m:=\left\lfloor d/s\right\rfloor .
$
Then \(m\le 2L\), and the residual rank \(d-ms\) is strictly smaller than
\(s\). Choose pairwise orthogonal projections
$
g_1,\ldots,g_m,g_0
$
such that \(g_1,\ldots,g_m\) have rank \(s\), \(g_0\) has rank \(d-ms\),
and
$
g_1+\cdots+g_m+g_0=1.
$
Let \(g\) be any rank-\(s\) projection. Then \(g_i\sim g\) for
\(1\le i\le m\). Also \(\operatorname{rank}(g_0)<s\le d/L\), so (4) gives
\[
|\tau_N(\varphi(g_0))|\le \frac1L+3\varepsilon .
\]
Using (1), \(\varepsilon\)-unitality, and (2), we get
\[
\left|
m\tau_N(\varphi(g))-1
\right|
\le
\frac1L+(3m+4)\varepsilon .
\]
Since
\[
\left|\frac1m-\frac{s}{d}\right|\le \frac1{L^2},
\]
and \(m\ge L\), it follows that
\[
\left|
\tau_N(\varphi(g))-\tau_M(g)
\right|
\le
\frac{2}{L^2}+5\varepsilon .
\tag{5}
\]

Let \(e\in M_d(\mathbb C)\) now be arbitrary. Decompose
$
e=e_1+\cdots+e_k+e_0,
$
where \(e_1,\ldots,e_k\) have rank \(s\), \(e_0\) has rank \(<s\), and
\(k\le m\le 2L\). By (5) and (2),
\[
\left|
\tau_N(\varphi(e_i))-\tau_M(e_i)
\right|
\le
\frac{2}{L^2}+7\varepsilon
\qquad (1\le i\le k).
\]
Moreover, by (4),
\[
|\tau_N(\varphi(e_0))|\le \frac1L+3\varepsilon,
\qquad
\tau_M(e_0)<\frac1L.
\]
Using (1) for the decomposition of \(e\), we obtain
\[
\begin{aligned}
\bigl|\tau_N(\varphi(e))-\tau_M(e)\bigr|
&\le
k\left(\frac{2}{L^2}+7\varepsilon\right)
+\frac2L+3\varepsilon+k\varepsilon  \\
&\le
\frac6L+(16L+3)\varepsilon .
\end{aligned}
\tag{6}
\]

\smallskip

\noindent
\emph{Case 3: \(M\) is diffuse.}
We use the standard fact that in a diffuse finite factor, projections of
arbitrary prescribed trace exist, and projections of equal trace are
Murray-von Neumann equivalent.

First let \(e\in M\) be a projection with \(\tau_M(e)\le 1/L\). Then we can
find pairwise orthogonal projections
$
e_1,\ldots,e_L
$
with \(e_1=e\) and \(e_i\sim e\) for every \(i\). The same argument as above
gives
\[
|\tau_N(\varphi(e))|\le \frac1L+3\varepsilon .
\tag{7}
\]

Now let \(e\in M\) be arbitrary and write \(t:=\tau_M(e)\). Choose
\(k\in\{0,\ldots,L\}\) such that
\[
\frac{k}{L}\le t<\frac{k+1}{L}.
\]
Choose a subprojection \(f\le e\) with
$
\tau_M(f)=\frac{k}{L},
$
and write
$
r:=e-f.
$
Then \(\tau_M(r)<1/L\). Let \(h_1,\ldots,h_L\) be pairwise orthogonal
equivalent projections summing to \(1\), and put
$
h:=h_1+\cdots+h_k .
$
Then \(f\sim h\). Arguing as in Case 1, but with \(L\) instead of \(d\), gives
\[
\bigl|\tau_N(\varphi(f))-\tau_M(f)\bigr|
\le (4L+2)\varepsilon .
\]
Using approximate additivity for \(e=f+r\), together with (7), we get
\[
\begin{aligned}
\bigl|\tau_N(\varphi(e))-\tau_M(e)\bigr|
&\le
\bigl|\tau_N(\varphi(f))-\tau_M(f)\bigr|
+|\tau_N(\varphi(r))|+\tau_M(r)+\varepsilon  \\
&\le
\frac2L+(4L+6)\varepsilon .
\end{aligned}
\tag{8}
\]

Combining (3), (6), and (8), we have shown that for every finite factor
\(M\), every projection \(e\in M\), and every integer \(L\ge 2\),
\[
\bigl|\tau_N(\varphi(e))-\tau_M(e)\bigr|
\le
\max\left\{
(4L^2+2)\varepsilon,\,
\frac6L+(16L+3)\varepsilon
\right\}.
\]
Thus every projection satisfies
\[
\bigl|\tau_N(\varphi(e))-\tau_M(e)\bigr|
\le \beta_p(\varepsilon),
\tag{9}
\]
where we define
\[
\beta_p(\varepsilon)
:=
\inf_{L\ge 2}
\max\left\{
(4L^2+2)\varepsilon,\,
\frac6L+(16L+3)\varepsilon
\right\}.
\]
Clearly \(\beta_p(\varepsilon)\to 0\) as \(\varepsilon\to 0\).

We now pass from projections to positive contractions. Let \(0\le a\le 1\)
in \(M\), and fix \(m\ge 1\). Choose a spectral step approximation
\[
b=\sum_{j=1}^m \lambda_jp_j,
\qquad
0\le \lambda_j\le 1,
\]
where the \(p_j\)'s are pairwise orthogonal projections summing to \(1\), and
$
\|a-b\|\le \frac1m.
$
Put \(h:=a-b\). Then \(h=m^{-1}c\) for some contraction \(c\in M_{\le 1}\).
Approximate homogeneity and \(\varphi(M_{\le 1})\subseteq N_{\le 1}\) give
\[
|\tau_N(\varphi(h))|\le \frac1m+\varepsilon.
\]
Approximate additivity gives
\[
|\tau_N(\varphi(a))-\tau_N(\varphi(b))|
\le \frac1m+2\varepsilon.
\]
On the other hand, by (1) and (9),
\[
\begin{aligned}
|\tau_N(\varphi(b))-\tau_M(b)|
&\le
(2m-1)\varepsilon
+
\sum_{j=1}^m |\lambda_j|
\bigl|\tau_N(\varphi(p_j))-\tau_M(p_j)\bigr| \\
&\le
(2m-1)\varepsilon+m\beta_p(\varepsilon).
\end{aligned}
\]
Since \(|\tau_M(a)-\tau_M(b)|\le 1/m\), we obtain
\[
|\tau_N(\varphi(a))-\tau_M(a)|
\le
\frac2m+(2m+1)\varepsilon+m\beta_p(\varepsilon).
\]
Taking the infimum over \(m\ge 1\), we get a function
\[
\beta_+(\varepsilon)\to 0
\quad(\varepsilon\to 0)
\]
such that
$
|\tau_N(\varphi(a))-\tau_M(a)|\le \beta_+(\varepsilon)
$
for every positive contraction \(a\in M\).

Let now \(x=x^*\in M_{\le 1}\). Write
\[
x=x_+-x_-,
\qquad
0\le x_\pm\le 1,
\qquad
x_+x_-=0.
\]
Approximate additivity and approximate homogeneity give
\[
\|\varphi(x)-\varphi(x_+)+\varphi(x_-)\|_2\le 2\varepsilon.
\]
Hence
\[
|\tau_N(\varphi(x))-\tau_M(x)|
\le 2\beta_+(\varepsilon)+2\varepsilon.
\]
Finally, for arbitrary \(x\in M_{\le 1}\), write
\[
x=a+ib,
\qquad
a=a^*,\quad b=b^*,\quad \|a\|,\|b\|\le 1.
\]
Approximate additivity and approximate homogeneity give
\[
\|\varphi(x)-\varphi(a)-i\varphi(b)\|_2\le 2\varepsilon.
\]
Therefore
\[
|\tau_N(\varphi(x))-\tau_M(x)|
\le 4\beta_+(\varepsilon)+6\varepsilon.
\]
Set
$
\beta(\varepsilon):=4\beta_+(\varepsilon)+6\varepsilon .
$
This proves the uniform trace estimate.
\end{proof}

\begin{corollary}
\label{cor:quasi-isometry-lipschitz}
There exists a function $\eta\colon [0,1]\to [0,\infty)$ with
$\eta(\varepsilon)\to 0$ as $\varepsilon\to 0$ such that the following
holds. Let $(N,\tau_N)$ be a finite factor, let $(M,\tau_M)$ be a tracial von
Neumann algebra, and let
$
\varphi\colon N_{\le 1}\to M_{\le 1}
$
be an $\varepsilon$-unital $L^2$-$\varepsilon$-$*$-homomorphism. Then for all
$x,y\in N_{\le 1}$ one has
\begin{align*}
\bigl|\|\varphi(x)\|_{2,\tau_M}-\|x\|_{2,\tau_N}\bigr|&\le \eta(\varepsilon),\\
\|\varphi(x)-\varphi(y)\|_{2,\tau_M}&\le \|x-y\|_{2,\tau_N}+\eta(\varepsilon).
\end{align*}
\end{corollary}

\begin{proof}
Let $\eta_0:[0,1]\to [0,\infty)$ be the modulus from
Proposition~\ref{P.uniform-trace-control-finite-factors}. Fix $x\in N_{\le 1}$. Since
$x^*x\in N_{\le 1}$, approximate multiplicativity and approximate
$*$-preservation give
\[
\begin{aligned}
\|\varphi(x)^*\varphi(x)-\varphi(x^*x)\|_{2,\tau_M}
&\le \|\varphi(x)^*\varphi(x)-\varphi(x^*)\varphi(x)\|_{2,\tau_M} \\
&\qquad +\|\varphi(x^*)\varphi(x)-\varphi(x^*x)\|_{2,\tau_M} \\
&\le 2\varepsilon.
\end{aligned}
\]
Therefore
\[
\begin{aligned}
\bigl|\|\varphi(x)\|_{2,\tau_M}^2-\|x\|_{2,\tau_N}^2\bigr|
&=\bigl|\tau_M(\varphi(x)^*\varphi(x)) - \tau_N(x^*x)\bigr| \\
&\le \|\varphi(x)^*\varphi(x)-\varphi(x^*x)\|_{2,\tau_M} \\
&\qquad + \bigl|\tau_M(\varphi(x^*x)) - \tau_N(x^*x)\bigr| \\
&\le 2\varepsilon+\eta_0(\varepsilon).
\end{aligned}
\]
Since $\|\varphi(x)\|_{2,\tau_M}\le 1$ and $\|x\|_{2,\tau_N}\le 1$, it follows that
\[
\bigl|\|\varphi(x)\|_{2,\tau_M}-\|x\|_{2,\tau_N}\bigr|
\le \bigl(2\varepsilon+\eta_0(\varepsilon)\bigr)^{1/2}.
\]

Now let $x,y\in N_{\le 1}$ and put
\[
s:=\frac{x+y}{2},
\qquad
h:=\frac{x-y}{2}.
\]
Then $s,h\in N_{\le 1}$, and the relations $x=s+h$ and $y=s-h$ hold inside
$N_{\le 1}$. Hence approximate additivity yields
\[
\|\varphi(x)-\varphi(s)-\varphi(h)\|_{2,\tau_M}\le \varepsilon
\]
and
\[
\|\varphi(y)-\varphi(s)-\varphi(-h)\|_{2,\tau_M}\le \varepsilon.
\]
Since approximate homogeneity gives
$
\|\varphi(-h)+\varphi(h)\|_{2,\tau_M}\le \varepsilon,
$
we obtain
\[
\|\varphi(x)-\varphi(y)-2\varphi(h)\|_{2,\tau_M}\le 3\varepsilon.
\]
Applying the first part to $h$ gives
\[
\begin{aligned}
\|\varphi(x)-\varphi(y)\|_{2,\tau_M}
&\le 2\|\varphi(h)\|_{2,\tau_M}+3\varepsilon \\
&\le 2\|h\|_{2,\tau_N}+2\bigl(2\varepsilon+\eta_0(\varepsilon)\bigr)^{1/2}+3\varepsilon \\
&= \|x-y\|_{2,\tau_N}+2\bigl(2\varepsilon+\eta_0(\varepsilon)\bigr)^{1/2}+3\varepsilon.
\end{aligned}
\]
Thus the assertions hold with
$
\eta(\varepsilon):=2\bigl(2\varepsilon+\eta_0(\varepsilon)\bigr)^{1/2}+3\varepsilon.
$
\end{proof}

\subsection{Ulam stability for finite-dimensional unitary groups} \label{sec:dCOT}

In this section we record a stability theorem for almost homomorphisms from ${\rm U}(n)$ to the unitary group of a tracial von Neumann algebra, with respect to the normalized trace norm. It is a standard variation of the de Chiffre--Ozawa--Thom argument~\cite{dCOT}: one uses a finite-valued measurable approximation under the $\varepsilon$-isometric hypothesis, and then applies Pettis' theorem~\cite{Pettis1950} to replace the resulting measurable positive definite kernel by a continuous one. The result may also be viewed as the tracial-von-Neumann analogue of the finite-group Hilbert--Schmidt stability theorem of Gowers and Hatami~\cite{Gowers_Hatami}.

\begin{definition}
Let $(M,\tau)$ be a tracial von Neumann algebra and let
$
\varphi\colon {\rm U}(n)\to \mathcal U(M)
$
be a map.
\begin{enumerate}
\item[(i)] We say that $\varphi$ is an $\varepsilon$-representation if
\[
\|\varphi(uv)-\varphi(u)\varphi(v)\|_{2,\tau}\le \varepsilon
\qquad (u,v\in {\rm U}(n)).
\]
\item[(ii)] We say that $\varphi$ is $\varepsilon$-isometric if
\[
\|\varphi(u)-\varphi(v)\|_{2,\tau}\le \|u-v\|_{2,n}+\varepsilon
\qquad (u,v\in {\rm U}(n)).
\]
\end{enumerate}
\end{definition}

Our main result will be proved after a sequence of lemmas. The first one is a measurable approximation result for $\varepsilon$-isometric $\varepsilon$-representations.

\begin{lemma}\label{app:lem-measurable-approx}
Let $(M,\tau)$ be a tracial von Neumann algebra and let
$
\varphi\colon {\rm U}(n)\to \mathcal U(M)
$
be an $\varepsilon$-isometric $\varepsilon$-representation. Then there exists a Borel map $\psi\colon {\rm U}(n)\to \mathcal U(M)$ with finite image such that
\[
\sup_{u\in {\rm U}(n)}\|\varphi(u)-\psi(u)\|_{2,\tau}\le 2\varepsilon
\]
and
\[
\|\psi(uv)-\psi(u)\psi(v)\|_{2,\tau}\le 7\varepsilon
\qquad (u,v\in {\rm U}(n)).
\]
\end{lemma}

\begin{proof}
Equip ${\rm U}(n)$ with the metric $d(u,v):=\|u-v\|_{2,n}$. Choose a finite $\varepsilon$-net $\{u_1,\dots,u_N\}\subseteq {\rm U}(n)$. Since ${\rm U}(n)$ is compact metric, there is a Borel partition
$
{\rm U}(n)=A_1\sqcup\cdots\sqcup A_N
$
with $A_i\subseteq B(u_i,\varepsilon)$. Define
\[
\psi(u):=\varphi(u_i)
\qquad(u\in A_i).
\]
Then $\psi$ is Borel and has finite image. For $u\in A_i$,
\[
\|\varphi(u)-\psi(u)\|_{2,\tau}=\|\varphi(u)-\varphi(u_i)\|_{2,\tau}\le \|u-u_i\|_{2,n}+\varepsilon\le 2\varepsilon.
\]
For the defect, use the unitary invariance of $\|\cdot\|_{2,\tau}$:
\begin{align*}
\|\psi(uv)-\psi(u)\psi(v)\|_{2,\tau}
&\le \|\psi(uv)-\varphi(uv)\|_{2,\tau}
   +\|\varphi(uv)-\varphi(u)\varphi(v)\|_{2,\tau} \\
&\qquad +\|\varphi(u)\varphi(v)-\psi(u)\psi(v)\|_{2,\tau}\\
&\le 2\varepsilon+\varepsilon+\|\varphi(u)-\psi(u)\|_{2,\tau}+\|\varphi(v)-\psi(v)\|_{2,\tau}\\
&\le 7\varepsilon.
\end{align*}
\end{proof}

\begin{definition}
Let $(M,\tau)$ be a tracial von Neumann algebra.
A function $\varphi\colon {\rm U}(n)\to M$ is positive definite if for every finite set $F\subseteq {\rm U}(n)$ the block matrix
\[
[\varphi(uv^{-1})]_{u,v\in F}\in M_{|F|}(M)
\]
is positive.
\end{definition}

When working on ${\rm U}(n)$, measurability refers to the Haar probability measure and the ultraweak Borel structure on $M$. The following is a version of Pettis' theorem~\cite{Pettis1950} for operator-valued positive definite functions.

\begin{proposition}\label{app:prop-pd-continuous}
Let $(M,\tau)$ be a tracial von Neumann algebra and let $\varphi\colon {\rm U}(n)\to M$ be measurable and positive definite. Then there exists a unique continuous positive definite function $\varphi_c\colon {\rm U}(n)\to M$ such that $\varphi=\varphi_c$ almost everywhere.
\end{proposition}

\begin{proof}
Fix a faithful normal representation $M\subseteq B(\mathcal H)$. For $\xi,\eta\in\mathcal H$ the scalar coefficient
$
u\mapsto \langle \varphi(u)\xi,\eta\rangle
$
is measurable. For $\xi=\eta$ it is also positive definite. By the classical theorem of Dixmier for scalar-valued measurable positive definite functions on locally compact groups, each diagonal coefficient admits a unique continuous positive definite representative; see \cite{Dixmier}*{\S13.4}. Polarization then yields continuous representatives for all coefficients. These assemble to an ultraweakly continuous map
$
\varphi_c\colon {\rm U}(n)\to B(\mathcal H)
$
with $\varphi=\varphi_c$ almost everywhere. Since a full-measure subset of ${\rm U}(n)$ is dense and $M$ is ultraweakly closed in $B(\mathcal H)$, it follows that $\varphi_c(u)\in M$ for every $u\in {\rm U}(n)$. Positivity of the finite block matrices passes to $\varphi_c$ by continuity, and uniqueness is immediate because two continuous functions on ${\rm U}(n)$ that agree almost everywhere agree everywhere.
\end{proof}

The main result of this section is the following theorem:

\begin{theorem}\label{app:thm-dCOT}
For every $\delta>0$, there exists $\varepsilon>0$ such that the following holds.
Let $(M,\tau)$ be a tracial von Neumann algebra and let
$
\varphi\colon {\rm U}(n)\to \mathcal U(M)
$
be an $\varepsilon$-isometric $\varepsilon$-representation. Then there exist
\begin{itemize}[leftmargin=2em]
\item a semifinite tracial von Neumann algebra $(P,\Tr)$ and a projection $p\in P$ with $\Tr(p)=1$ and $pPp\cong M$;
\item a projection $q\in P$ with $\Tr(q)\in[1-\delta,1+\delta]$;
\item and a continuous group homomorphism
\[
\rho\colon {\rm U}(n)\to \mathcal U(qPq),
\]
where continuity is with respect to $\|\cdot\|_{2,\Tr}$ on $\mathcal U(qPq)$.
\end{itemize}
such that, after identifying $M$ with $pPp$, one has
\[
\sup_{u\in {\rm U}(n)}\|\varphi(u)-\rho(u)\|_{2,\Tr}<\delta.
\]
Here the difference is taken in the common ambient algebra $P$, so $qPq$ is an amplification corner of $M$.
\end{theorem}

\begin{proof}
Fix $\delta>0$. The proof is the same as in
\cite{dCOT}*{Theorem~5.2 and the proof of Theorem~1.6} once one inserts
the two preparatory steps that are specific to the compact group ${\rm U}(n)$.
Apply Lemma~\ref{app:lem-measurable-approx}.
This gives a Borel map
$
\psi\colon {\rm U}(n)\to \mathcal U(M)
$
with finite image such that
\[
\sup_{u\in {\rm U}(n)}\|\varphi(u)-\psi(u)\|_{2,\tau}\le 2\varepsilon
\]
and
\[
\|\psi(uv)-\psi(u)\psi(v)\|_{2,\tau}\le 7\varepsilon
\qquad(u,v\in {\rm U}(n)).
\]
Thus, after choosing $\varepsilon$ sufficiently small in terms of $\delta$, we may replace $\varphi$ by a measurable approximate
representation $\psi$ with the same type of small-defect estimates as in
\cite{dCOT}. Since the replacement is uniform, any final estimate for
$\psi$ transfers back to $\varphi$ by the triangle inequality.

Second, as in \cite{dCOT} one averages $\psi$ over ${\rm U}(n)$ to obtain an
$M$-valued positive definite kernel. In the discrete amenable setting this
object is defined on the nose everywhere, whereas here Haar averaging gives
only a measurable positive definite function. At this point
Proposition~\ref{app:prop-pd-continuous} enters: the averaged kernel has a
unique continuous positive definite representative. Since the two functions
agree almost everywhere, all inequalities established in \cite{dCOT} for the
averaged kernel remain valid for this continuous representative.

After this replacement, the operator-algebraic part of the proof of
\cite{dCOT}*{Theorem~5.2} applies word for word. One performs the Hilbert
$M$-module GNS/Stinespring construction for the continuous positive definite
kernel and obtains a semifinite tracial von Neumann algebra $(P,\Tr)$, a
projection $p\in P$ with $\Tr(p)=1$ and $pPp\cong M$, a projection $q\in P$
with $\Tr(q)\in[1-\delta,1+\delta]$, and a continuous homomorphism
$
\rho\colon {\rm U}(n)\to \mathcal U(qPq)
$
such that $\psi(u)$ is uniformly close to $\rho(u)$ in $\|\cdot\|_{2,\Tr}$,
after identifying $M$ with $pPp$.

Finally, the passage from $\psi$ back to $\varphi$ is exactly the same as in
the proof of \cite{dCOT}*{Theorem~1.6}. Choosing the initial parameters
sufficiently small yields
\[
\sup_{u\in {\rm U}(n)}\|\varphi(u)-\rho(u)\|_{2,\Tr}<\delta.
\]
This proves the theorem.
\end{proof}

\section{Ulam stability for matrix algebras}
\label{sec:matrix-stability}

The proof of the Ulam stability result for matrix algebras proceeds in four steps.

\emph{Step 0:} We first show that an $\varepsilon$-unital $L^2$-$\varepsilon$-$*$-homomorphism automatically preserves the trace approximately; see Proposition~\ref{P.uniform-trace-control-finite-factors}.

\emph{Step 1:} We show that an $L^2$-$\varepsilon$-$*$-homomorphism is close on the unitary group to a map into the unitary group which is approximately multiplicative; see Lemma~\ref{lem:matrix-unitary-restriction}.

\emph{Step 2:} Then, we apply Theorem \ref{app:thm-dCOT} to replace this approximate representation of ${\rm U}(n)$ by a genuine continuous representation after a small amplification.

\emph{Step 3:}  We then use a direct representation-theoretic argument to show that an almost isometric representation of ${\rm U}(n)$ is close to a multiple of the standard representation; this is Theorem \ref{T.standard-representation-extraction}.  This part turned out to be more subtle than we expected and relies on eigenvalue estimates for the Casimir operator, see Lemma \ref{app:lem-casimir-gap}.

\emph{Step 4:} Finally, we pass from unitaries to the whole unit ball by a standard extension argument, writing each contraction as a linear combination of four unitaries; see the proof of Theorem \ref{thm:matrix-stability-reduction}.

This strategy is inspired by Farah's work on rigidity phenomena for corona algebras, especially~\cite[Chapter 17]{Farah-book}, but the details are different because of the change of metric and the need to control the matrix size in a uniform way. In particular, the third step needs some new input.

\medskip

We now implement this program.

\begin{lemma}\label{lem:matrix-unitary-restriction}
There exists a function $\gamma : [0,1]\to [0,\infty)$ with
$\gamma(\varepsilon)\to 0$ as $\varepsilon\to 0$ such that the following
holds. Let $(N,\tau_N)$ and $(M,\tau_M)$ be tracial von Neumann algebras and
let
$
   \varphi:N_{\le 1}\to M_{\le 1}
$
be an $\varepsilon$-unital $L^2$-$\varepsilon$-$*$-homomorphism. Then, for
every $u\in \mathcal U(N)$, there exists a unitary
$
    \varphi_0(u)\in \mathcal U(M)
$
such that
$
   \|\varphi(u)-\varphi_0(u)\|_{2,\tau_M}\leq \gamma(\varepsilon).
$
Moreover, the map $\varphi_0:\mathcal U(N)\to \mathcal U(M)$ can be chosen so that
\[
    \sup_{u,v\in \mathcal U(N)}
   \|\varphi_0(uv)-\varphi_0(u)\varphi_0(v)\|_{2,\tau_M}
   \leq \gamma(\varepsilon).
\]
\end{lemma}

\begin{proof}
We use the following standard consequence of the polar decomposition in
tracial von Neumann algebras. If \((P,\tau_P)\) is tracial and \(x\in P\)
satisfies
\[
   \|x^*x-1\|_{2,\tau_P}\leq \delta,
   \qquad
   \|xx^*-1\|_{2,\tau_P}\leq \delta,
\]
then there are a unitary \(w\in P\) and a function
\(\alpha(\delta)\to 0\) such that
\[
   \|x-w\|_{2,\tau_P}\leq \alpha(\delta).
\]
This follows from the polar decomposition and the comparison theory for
projections in finite von Neumann algebras; see
\cite[Section V.I]{Takesaki-1}.

Let \(u\in \mathcal U(N)\). Since \(\varphi\) is approximately
\(*\)-preserving, approximately multiplicative, approximately unital, and
takes values in \(M_{\leq 1}\), we have
\[
\begin{aligned}
 \|\varphi(u)^*\varphi(u)-1\|_{2,\tau_M}
 &\leq
   \|\varphi(u)^*\varphi(u)-\varphi(u^*)\varphi(u)\|_{2,\tau_M}  \\
 &\quad+
   \|\varphi(u^*)\varphi(u)-\varphi(u^*u)\|_{2,\tau_M}
   +\|\varphi(1)-1\|_{2,\tau_M}                              \\
 &\leq 3\varepsilon .
\end{aligned}
\]
Similarly,
\[
   \|\varphi(u)\varphi(u)^*-1\|_{2,\tau_M}\leq 3\varepsilon .
\]
By the observation from the beginning of the proof, we get a map
\[
   \varphi_0:\mathcal U(N)\to \mathcal U(M)
\]
such that
\[
   \|\varphi(u)-\varphi_0(u)\|_{2,\tau_M}
   \leq \alpha(3\varepsilon).
\]

It remains to check approximate multiplicativity. Let
\(u,v\in \mathcal U(N)\). Then
\[
\begin{aligned}
 \|\varphi_0(uv)-\varphi_0(u)\varphi_0(v)\|_{2,\tau_M}
 &\leq
   \|\varphi_0(uv)-\varphi(uv)\|_{2,\tau_M} \\
 &\quad+
   \|\varphi(uv)-\varphi(u)\varphi(v)\|_{2,\tau_M} \\
 &\quad+
   \|\varphi(u)\varphi(v)-\varphi_0(u)\varphi_0(v)\|_{2,\tau_M}.
\end{aligned}
\]
The first term is at most \(\alpha(3\varepsilon)\), and the second term is at
most \(\varepsilon\). For the last term, write
\[
\begin{aligned}
\varphi(u)\varphi(v)-\varphi_0(u)\varphi_0(v)
&=
\bigl(\varphi(u)-\varphi_0(u)\bigr)\varphi(v) \\
&\quad+
\varphi_0(u)\bigl(\varphi(v)-\varphi_0(v)\bigr).
\end{aligned}
\]
Since \(\varphi(v)\in M_{\leq 1}\), we have
\(\|\varphi(v)\|_\infty\leq 1\). Also, since \(\varphi_0(u)\) is unitary,
left multiplication by \(\varphi_0(u)\) preserves the \(L^2\)-norm. Hence
\[
\begin{aligned}
&\|\varphi(u)\varphi(v)-\varphi_0(u)\varphi_0(v)\|_{2,\tau_M} \\
&\qquad\leq
   \|\bigl(\varphi(u)-\varphi_0(u)\bigr)\varphi(v)\|_{2,\tau_M}
   +
   \|\varphi_0(u)\bigl(\varphi(v)-\varphi_0(v)\bigr)\|_{2,\tau_M} \\
&\qquad\leq
   \|\varphi(u)-\varphi_0(u)\|_{2,\tau_M}
   +
   \|\varphi(v)-\varphi_0(v)\|_{2,\tau_M} \\
&\qquad\leq
   2\alpha(3\varepsilon).
\end{aligned}
\]
Therefore
\[
\begin{aligned}
 \|\varphi_0(uv)-\varphi_0(u)\varphi_0(v)\|_{2,\tau_M}
 &\leq
   \alpha(3\varepsilon)+\varepsilon+2\alpha(3\varepsilon) \\
 &= \varepsilon+3\alpha(3\varepsilon).
\end{aligned}
\]
Define
\[
   \gamma(\varepsilon)
   :=
   \varepsilon+3\alpha(3\varepsilon),
   \qquad 0\leq \varepsilon\leq 1.
\]
Then \(\gamma(\varepsilon)\to 0\) as \(\varepsilon\to 0\). Moreover, for every
\(u\in \mathcal U(N)\),
\[
   \|\varphi(u)-\varphi_0(u)\|_{2,\tau_M}
   \leq \alpha(3\varepsilon)
   \leq \gamma(\varepsilon),
\]
and
\[
   \sup_{u,v\in \mathcal U(N)}
   \|\varphi_0(uv)-\varphi_0(u)\varphi_0(v)\|_{2,\tau_M}
   \leq \gamma(\varepsilon).
\]
This proves the lemma.
\end{proof}

\begin{corollary}\label{C.matrix-unitary-correction}\label{cor:matrix-unitary-correction}
There exists a function $\beta\colon [0,1]\to [0,\infty)$ with $\beta(\varepsilon)\to0$ as $\varepsilon\to0$ such that the following holds for every $n\in\bbN$ and every tracial von Neumann algebra $(M,\tau_M)$. Let
$
\varphi\colon {\rm M}_{n}(\mathbb C)_{\le 1}\to M_{\le 1}
$
be an $\varepsilon$-unital $L^2$-$\varepsilon$-$*$-homomorphism. Then there exist
\begin{itemize}[leftmargin=2em]
\item a semifinite tracial von Neumann algebra $(P,\Tr)$;
\item finite projections $p,e\in P$ with
\[
\Tr(p)=1,
\qquad
pPp\cong M,
\qquad
\Tr(e)\in[1-\beta(\varepsilon),1+\beta(\varepsilon)];
\]
\item and a continuous homomorphism
$
\pi\colon {\rm U}(n)\to \mathcal U(ePe)
$
\end{itemize}
such that, after identifying $M$ with $pPp$, one has
\[
\sup_{u\in{\rm U}(n)}
\|\varphi(u)-\pi(u)\|_{2,\Tr}
\le \beta(\varepsilon), \quad
\|e-p\|_{2,\Tr}\le \beta(\varepsilon),
\]
\[
\sup_{z\in\mathbb T}
\|\pi(zI_n)-ze\|_{2,\Tr}
\le \beta(\varepsilon),
\]
and
\[
\sup_{u\in{\rm U}(n)}
\left|
\Tr(\pi(u))-\tau_n(u)
\right|
\le \beta(\varepsilon).
\]
\end{corollary}

\begin{proof}
Let $\beta_1$ be the modulus from Lemma~\ref{lem:matrix-unitary-restriction}, let $\beta_2$ be the modulus from Corollary~\ref{cor:quasi-isometry-lipschitz}, and let $\beta_3$ be the modulus from Proposition~\ref{P.uniform-trace-control-finite-factors}. By Theorem~\ref{app:thm-dCOT}, we also get a function
\[
\beta_4\colon [0,\infty)\to [0,\infty)
\]
with $\beta_4(\eps)\to0$ as $\eps\to0$ such that every $\eps$-isometric $\eps$-representation of ${\rm U}(n)$ admits the conclusion of Theorem~\ref{app:thm-dCOT} with error bound $\beta_4(\eps)$.

Define
\[
\beta_5(\varepsilon):=2\beta_2(\varepsilon)+4\varepsilon+2\beta_1(\varepsilon),
\qquad
\beta_6(\varepsilon):=\beta_1(\varepsilon)+\beta_4\bigl(\beta_5(\varepsilon)\bigr),
\]
and
\[
\beta(\varepsilon)
:=
\max\Biggl\{
\begin{aligned}
&\beta_6(\varepsilon)+\varepsilon,
\qquad 2\beta_6(\varepsilon)+\varepsilon,
\qquad (\beta_6(\varepsilon)+\varepsilon)^2, \\
&\beta_6(\varepsilon)+\beta_3(\varepsilon)
+(\beta_6(\varepsilon)+\varepsilon)
\bigl(1+(\beta_6(\varepsilon)+\varepsilon)^2\bigr)^{1/2}
\end{aligned}
\Biggr\}.
\]
Then $\beta(\varepsilon)\to0$ as $\varepsilon\to0$.

By Lemma~\ref{lem:matrix-unitary-restriction}, we may choose a map $u\mapsto v_u\in \mathcal U(M)$ such that
\[
\sup_{u\in {\rm U}(n)}\|\varphi(u)-v_u\|_{2,\tau_M}\le \beta_1(\varepsilon)
\]
and
\[
\sup_{u,w\in {\rm U}(n)}\|v_{uw}-v_uv_w\|_{2,\tau_M}\le \beta_1(\varepsilon).
\]
For $u,w\in {\rm U}(n)$, approximate additivity and scalar homogeneity give
\[
\Bigl\|2\varphi\Bigl(\frac{u-w}{2}\Bigr)-\bigl(\varphi(u)-\varphi(w)\bigr)\Bigr\|_{2,\tau_M}\le 4\varepsilon.
\]
Since $(u-w)/2$ is a contraction, Corollary~\ref{cor:quasi-isometry-lipschitz} yields
\[
\big|\|\varphi(u)-\varphi(w)\|_{2,\tau_M}-\|u-w\|_{2,n}\big|
\le 4\varepsilon+2\beta_2(\varepsilon)
\qquad (u,w\in {\rm U}(n))
\]
because
\[
\begin{aligned}
\big|\|\varphi(u)-\varphi(w)\|_{2,\tau_M}-\|u-w\|_{2,n}\big|
&\le
\Biggl|\|\varphi(u)-\varphi(w)\|_{2,\tau_M}
-2\Bigl\|\varphi\Bigl(\frac{u-w}{2}\Bigr)\Bigr\|_{2,\tau_M}\Biggr| \\
&\quad+
2\Biggl|\Bigl\|\varphi\Bigl(\frac{u-w}{2}\Bigr)\Bigr\|_{2,\tau_M}
-\Bigl\|\frac{u-w}{2}\Bigr\|_{2,n}\Biggr| \\
&\le 4\varepsilon+2\beta_2(\varepsilon).
\end{aligned}
\]
Hence
\[
\|v_u-v_w\|_{2,\tau_M}\le \|u-w\|_{2,n}+\beta_5(\varepsilon)
\qquad (u,w\in {\rm U}(n)).
\]
Thus the map $u\mapsto v_u$ is a $\beta_5(\varepsilon)$-isometric $\beta_5(\varepsilon)$-representation. Applying Theorem~\ref{app:thm-dCOT}, we obtain $(P,\Tr)$, projections $p,e\in P$, and a continuous homomorphism
$
\pi\colon {\rm U}(n)\to \mathcal U(ePe)
$
with $\Tr(p)=1$, $\Tr(e)\in[1-\beta_4(\beta_5(\varepsilon)),1+\beta_4(\beta_5(\varepsilon))]$, $pPp\cong M$, and
\[
\sup_{u\in {\rm U}(n)}\|v_u-\pi(u)\|_{2,\Tr}\le \beta_4\bigl(\beta_5(\varepsilon)\bigr).
\]
Therefore
\[
\sup_{u\in {\rm U}(n)}\|\varphi(u)-\pi(u)\|_{2,\Tr}
\le \beta_6(\varepsilon)
\le \beta(\varepsilon),
\]
and also
\[
\Tr(e)\in[1-\beta(\varepsilon),1+\beta(\varepsilon)].
\]

Since $e=\pi(1)$, $\varphi$ is $\varepsilon$-unital, and
$\|\pi(1)-\varphi(1)\|_{2,\Tr}\le\beta_6(\varepsilon)$, we have
\[
\|e-p\|_{2,\Tr}\le \beta_6(\varepsilon)+\varepsilon\le \beta(\varepsilon).
\]
In particular,
$
|\Tr(e)-1|\le \|e-p\|_{2,\Tr}^2
$,
so
$
\Tr(e)\in[1-\beta(\varepsilon),1+\beta(\varepsilon)].
$

The scalar-circle estimate is now obtained directly:
\[
\begin{aligned}
\|\pi(zI_n)-ze\|_{2,\Tr}
&\le
\|\pi(zI_n)-\varphi(zI_n)\|_{2,\Tr}  \\
&\quad+
\|\varphi(zI_n)-z\varphi(1)\|_{2,\Tr}
+
\|z\varphi(1)-ze\|_{2,\Tr}  \\
&\le 2\beta_6(\varepsilon)+\varepsilon
\le \beta(\varepsilon).
\end{aligned}
\]

For $v\in{\rm U}(n)$, the estimate
$
\|\pi(v)-\varphi(v)\|_{2,\Tr}\le \beta_6(\varepsilon)
$
implies
\[
\begin{aligned}
\left|
\Tr\bigl(p\pi(v)\bigr)-\tau_M\bigl(\varphi(v)\bigr)
\right|
&=\left|\Tr\bigl((\pi(v)-\varphi(v))p\bigr)\right| \\
&\le \|\pi(v)-\varphi(v)\|_{2,\Tr} \\
&\le \beta_6(\varepsilon).
\end{aligned}
\]
By Proposition~\ref{P.uniform-trace-control-finite-factors},
\[
\left|
\tau_M\bigl(\varphi(v)\bigr)-\tau_n(v)
\right|
\le \beta_3(\varepsilon),
\]
we get
\[
\left|
\Tr\bigl(p\pi(v)\bigr)-\tau_n(v)
\right|
\le \beta_6(\varepsilon)+\beta_3(\varepsilon)
\qquad (v\in {\rm U}(n)).
\]
Since $\pi(v)\in ePe$ is unitary in $ePe$, Cauchy's inequality gives
\[
\begin{aligned}
\left|
\Tr(\pi(v)) - \Tr\bigl(p\pi(v)\bigr)
\right|
&=
\left|\Tr\bigl((e-p)\pi(v)\bigr)\right| \\
&\le
\|e-p\|_{2,\Tr}\,\|\pi(v)\|_{2,\Tr} \\
&\le
\bigl(\beta_6(\varepsilon)+\varepsilon\bigr)
\bigl(1+(\beta_6(\varepsilon)+\varepsilon)^2\bigr)^{1/2},
\end{aligned}
\]
and therefore
$
\left|
\Tr(\pi(v))-\tau_n(v)
\right|
\le \beta(\varepsilon).
$
This proves the corollary.
\end{proof}

For each $n$, fix a Hilbert--Schmidt orthonormal basis
$
\mathcal X_n\subset \mathfrak{su}(n)
$
whose elements satisfy $\|X\|\le 2$; such a basis can be obtained from the generalized Weyl unitaries by taking suitable skew-adjoint real and imaginary parts.

We proceed with the eigenvalue estimates for the Casimir operator.

\begin{lemma}\label{app:lem-casimir-gap}
Let $\sigma:{\rm U}(n)\to {\rm U}(d)$ be an irreducible representation on which
the center acts by the defining character $z\mapsto z$, and let
$\rho_\sigma=d\sigma_e|_{\mathfrak{su}(n)}$ be the induced
representation of $\mathfrak{su}(n)$. For any Hilbert--Schmidt
orthonormal basis $\{X_1,\dots,X_{n^2-1}\}$ of $\mathfrak{su}(n)$,
the operator
\[
C_\sigma:=\frac{1}{n^2-1}\sum_{j=1}^{n^2-1}\rho_\sigma(X_j)^*\rho_\sigma(X_j)
\]
is a scalar multiple of the identity. That scalar is equal to $1$ for
the standard representation, and is at least $2$ for every other such
irreducible representation.
\end{lemma}

\begin{proof}
The operator \(C_\sigma\) is the image of the normalized quadratic Casimir
element for \(\mathfrak{su}(n)\). Hence it commutes with the irreducible
\(\SU(n)\)-representation \(\sigma|_{\SU(n)}\), and is therefore scalar by
Schur's lemma.

Let
\[
\lambda=(\lambda_1,\ldots,\lambda_n),\qquad
\lambda_1\geq\cdots\geq \lambda_n,
\]
be the highest weight of $\sigma$. Put $k:=\lambda_n$ and
$
\mu_i:=\lambda_i-\lambda_n.
$
Then $\mu_n=0$, $\mu_1\geq\cdots\geq\mu_n=0$, and
$
V_\lambda \cong (\det)^k\otimes V_\mu
$ can be seen easily.
The restriction to ${\rm SU}(n)$, and hence the value of the Casimir operator on
$\mathfrak{su}(n)$, depends only on $\mu$.

The assumption that the center acts by the defining character means that
$
\sigma(zI_n)=zI$ for all $z\in\mathbb T$.
Equivalently,
\[
|\lambda|:=\lambda_1+\cdots+\lambda_n=1.
\]
Since
\[
|\mu|=|\lambda|-nk=1-nk,
\]
we have
$
|\mu|\equiv 1 \pmod n.
$
Thus either $|\mu|=1$, in which case
$
\mu=(1,0,\ldots,0)
$
and $\sigma$ is the standard representation, or else
$
|\mu|\geq n+1.
$

By the highest-weight formula for the Casimir eigenvalue;
see \cite[Chapter~V, Proposition~5.28(b)]{KnappLieGroups}, the Casimir operator acts on the irreducible
\(\SU(n)\)-module of highest weight \(\mu\) by
\[
   \langle \tilde\mu,\tilde\mu+2\rho\rangle ,
   \qquad
   \tilde\mu=\mu-\frac{|\mu|}{n}(1,\ldots,1).
\]
With our Hilbert--Schmidt normalization and the additional factor
\((n^2-1)^{-1}\), this becomes
\[
c(\mu)
=
\frac{n}{n^2-1}
\left(
   n|\mu|
   +
   \sum_{j=1}^n \mu_j(\mu_j+1-2j)
   -
   \frac{|\mu|^2}{n}
\right).
\]

For $\mu=(1,0,\ldots,0)$, this gives $c(\mu)=1$.

It remains to show that $c(\mu)\geq 2$ for every other partition
$\mu$ with $\mu_n=0$ and $|\mu|\equiv 1\pmod n$.
We first minimize the expression for fixed $S:=|\mu|$. Since the terms
$nS-S^2/n$ are fixed, it suffices to minimize
\[
F(\mu):=\sum_{j=1}^n \mu_j(\mu_j+1-2j).
\]
Suppose that for some $i<n-1$ one has $\mu_i\geq \mu_{i+1}+2$.
Replace $(\mu_i,\mu_{i+1})$ by $(\mu_i-1,\mu_{i+1}+1)$. This preserves
the condition of being a partition and changes $F$ by
\[
\begin{aligned}
&\bigl((\mu_i-1)(\mu_i-2i)-\mu_i(\mu_i+1-2i)\bigr) \\
&\quad+
\bigl((\mu_{i+1}+1)(\mu_{i+1}-2i)-\mu_{i+1}(\mu_{i+1}-2i-1)\bigr) \\
&=
2(\mu_{i+1}-\mu_i+2)\leq 0.
\end{aligned}
\]
It is strictly negative unless $\mu_i=\mu_{i+1}+2$. Iterating, we see
that at a minimum for fixed $S$, the first $n-1$ entries differ by at
most $1$. Therefore, for fixed $S$, the minimum is attained when the
first $n-1$ entries of $\mu$ are as balanced as possible.

Write
\[
S=q(n-1)+r,\qquad 0\leq r\leq n-2 .
\]
The balanced partition is
\[
\mu^{(S)}
=
(\underbrace{q+1,\ldots,q+1}_{r\ \mathrm{times}},
\underbrace{q,\ldots,q}_{n-1-r\ \mathrm{times}},
0).
\]
A direct substitution into the Casimir formula gives
\[
c(\mu^{(S)})
=
\frac{(n+q-r)(nq+nr-q+r)}{n^2-1}.
\]
Indeed, under the congruence condition $r-q\equiv 1\pmod n$, the
expression
\[
\frac{(n+q-r)(nq+nr-q+r)}{n^2-1}
\]
is minimized, among all non-standard admissible pairs $(q,r)$, by
$
(q,r)=(3,0)$ for $n=2$,
by
$
(q,r)=(2,0)$ for $n=3$,
and by
$
(q,r)=(1,2)$ for $n \geq 4$.
This gives respectively $5$, $5/2$, and $(3n+1)/(n+1)>2$.

Now impose $S\equiv 1\pmod n$. Since $S=q(n-1)+r$, this is equivalent
to
\[
r-q\equiv 1\pmod n .
\]
The case $S=1$ is $q=0$, $r=1$, giving the standard representation.

For the remaining cases, the smallest possible $S$ is as follows. If
$n=2$, then $S=3$, and the minimizing partition is $(3,0)$, giving
$
c(3,0)=5.
$
If $n=3$, then $S=4$, and the minimizing partition is $(2,2,0)$, giving
$
c(2,2,0)=\frac52.
$
If $n\geq 4$, then the smallest possible $S>1$ is $S=n+1$, with
balanced partition
$
\mu=(2,2,1,\ldots,1,0),
$
where the entry $1$ occurs $n-3$ times. In this case
\[
c(\mu)=\frac{3n+1}{n+1}=3-\frac{2}{n+1}>2.
\]
For larger admissible values of $S$, the balanced value above is larger,
and every other partition of the same size has the Casimir value which is at least the balanced value. Therefore every non-standard irreducible representation with central character $z\mapsto z$ has normalized Casimir value at least $2$. This proves the lemma.
\end{proof}

\begin{theorem}\label{T.standard-representation-extraction}\label{thm:matrix-rep}
There exists a function $\xi\colon [0,1]\to [0,\infty)$ with $\xi(\varepsilon)\to 0$ as $\varepsilon\to 0$ such that the following holds for every $n\in\bbN$. Let $(P,\Tr)$ be a semifinite tracial von Neumann algebra, let $e\in P$ be a finite projection, and let
$
\pi\colon {\rm U}(n)\to \mathcal U(ePe)
$
be a continuous group homomorphism. Assume that
\[
\sup_{z\in\mathbb T}
\|\pi(zI_n)-ze\|_{2,\Tr}\le\varepsilon
\]
and
\[
\sup_{u\in{\rm U}(n)}
\left|
\Tr(\pi(u))-\tau_n(u)
\right|
\le\varepsilon.
\]
Then there exist a finite projection $q\le e$ and a $*$-homomorphism
$
\psi\colon {\rm M}_{n}(\mathbb C)\to qPq
$
with $\psi(1)=q$ such that
\[
\sup_{u\in {\rm U}(n)}\|\pi(u)-\psi(u)\|_{2,\Tr}\le \xi(\varepsilon).
\]
\end{theorem}

\begin{proof}
Let $N:=W^*(\pi({\rm U}(n)))\subseteq ePe$. By the Peter--Weyl theorem, there are central projections $Q_\sigma\in N$, indexed by the irreducible representations $\sigma\in\widehat{{\rm U}(n)}$, such that on the $\sigma$-isotypic summand one has
\[
Q_\sigma N\cong B(H_\sigma)\overline\otimes N_\sigma,
\qquad
\pi(u)Q_\sigma\cong \sigma(u)\otimes 1
\qquad (u\in {\rm U}(n)).
\]
Set
$
p_\sigma:=\Tr(Q_\sigma).
$
Then
\[
\Tr\bigl(\pi(u)\bigr)
=
\sum_\sigma p_\sigma\frac{\chi_\sigma(u)}{\dim\sigma}.
\]

Let $Q_1$ be the projection onto the sum of those irreducible summands on which the center of ${\rm U}(n)$ acts by the defining character $z\mapsto z$, and let $Q_{\mathrm{st}}\le Q_1$ be the projection onto the standard isotypic component.

We first control the part outside $Q_1$. Since $\pi(zI_n)$ is central, the contribution of an irreducible summand with central character $z\mapsto z^k$ to $\|\pi(zI_n)-ze\|_{2,\Tr}^2$ is $|z^k-z|^2$. Integrating over $z\in\mathbb T$, we get
\[
2\Tr\bigl(e(1-Q_1)\bigr)\le\varepsilon^2.
\]
Thus, with
$
p_0:=\Tr\bigl(e(1-Q_1)\bigr),
$
we have
$
p_0\le \varepsilon^2/2.
$

We now control the part in $Q_1-Q_{\mathrm{st}}$. We use the heat semigroup on ${\rm SU}(n)$ associated with the normalized Casimir operator
\[
C_\sigma=
\frac{1}{n^2-1}
\sum_{X\in \mathcal X_n} d\sigma(X)^*d\sigma(X),
\]
where $\mathcal X_n\subset \mathfrak{su}(n)$ is the fixed Hilbert--Schmidt orthonormal basis from the beginning of this subsection. Let $h_t$ denote the corresponding heat kernel at some fixed time $t>0$. For an irreducible summand $\sigma$, write $c_\sigma$ for its normalized Casimir eigenvalue. We use the standard Peter--Weyl expansion of the heat kernel on a compact
Lie group; see Fegan~\cite[Section~4, especially (4.6)]{Fegan1978}.
With our normalization of the bi-invariant Laplacian, the \(\sigma\)-isotypic
component is an eigenspace with eigenvalue \(-c_\sigma\), where \(c_\sigma\)
is the corresponding normalized Casimir eigenvalue. Hence, if \(h_t\) denotes
the heat kernel with respect to Haar probability measure, then
\[
   h_t(s)=\sum_{\sigma\in\widehat{\mathrm{SU}(n)}}
      \dim(\sigma)\chi_\sigma(s)e^{-tc_\sigma},
\]
and character orthogonality gives
\[
   \int_{\mathrm{SU}(n)}
      \frac{\chi_\sigma(s)}{\dim\sigma}\,h_t(s)\,ds
   =
   e^{-tc_\sigma}.
\]
For the standard representation $c_{\mathrm{st}}=1$, while by Lemma~\ref{app:lem-casimir-gap}, every non-standard irreducible summand with central character $z\mapsto z$ satisfies $c_\sigma\ge 2$.

Applying the character estimate to \(s\in {\rm SU}(n)\) and integrating
against \(h_t(s)\,ds\) and since
\[
\int_{{\rm SU}(n)}\tau_n(s)h_t(s)\,ds=e^{-t},
\]
we obtain
\[
\left|
   \sum_\sigma p_\sigma e^{-tc_\sigma}
   -
   e^{-t}
\right|
\le \varepsilon.
\]
Since the same character estimate at \(u=1\) gives
$
   |\operatorname{Tr}(e)-1|\le\varepsilon,
$
we also have
\[
\left|
   \sum_\sigma p_\sigma e^{-tc_\sigma}
   -
   \operatorname{Tr}(e)e^{-t}
\right|
\le
\varepsilon+\varepsilon e^{-t}
\le 2\varepsilon.
\]
Let
\[
   p_{\rm st}:=\operatorname{Tr}(Q_{\rm st}),\qquad
   p_{\rm bad}:=\operatorname{Tr}(Q_1-Q_{\rm st}).
\]
Since
$
   \operatorname{Tr}(e)=p_{\rm st}+p_{\rm bad}+p_0,
$
and since \(c_{\rm st}=1\), while \(c_\sigma\ge2\) on
\(Q_1-Q_{\rm st}\), we get
\[
   \sum_\sigma p_\sigma e^{-tc_\sigma}
   \le
   p_{\rm st}e^{-t}+p_{\rm bad}e^{-2t}+p_0.
\]
Therefore
\[
   p_{\rm bad}(e^{-t}-e^{-2t})
   \le
   2\varepsilon+p_0(1-e^{-t}).
\]
Choosing \(t=1\), we get \(p_{\rm bad}\le C\varepsilon\).
Since
$
p_0\le \varepsilon^2/2,
$
it follows that
\[
\Tr(e-Q_{\mathrm{st}})
=
p_0+p_{\mathrm{bad}}
\le C'\varepsilon.
\]

On the standard isotypic component one has
$
Q_{\mathrm{st}}N\cong {\rm M}_n(\mathbb C)\overline\otimes N_{\mathrm{st}}
$
for some von Neumann algebra $N_{\mathrm{st}}$, and under this identification
\[
\pi(u)Q_{\mathrm{st}}\cong u\otimes 1
\qquad (u\in {\rm U}(n)).
\]
Let $q:=Q_{\mathrm{st}}$ and let
$
\psi\colon {\rm M}_n(\mathbb C)\to qPq
$
be the corresponding inclusion of the first tensor factor. Then $\psi$ is a unital $*$-homomorphism and
since $q\le e$ and $e$ is finite, the projection $q$ is finite. Moreover,
\[
\psi(u)=\pi(u)Q_{\mathrm{st}}
\qquad (u\in {\rm U}(n)).
\]
Therefore
\[
\pi(u)-\psi(u)=\pi(u)(e-Q_{\mathrm{st}})
\qquad (u\in {\rm U}(n)),
\]
and hence
\[
\|\pi(u)-\psi(u)\|_{2,\Tr}^2
=
\Tr(e-Q_{\mathrm{st}})
\le C'\varepsilon
\qquad (u\in {\rm U}(n)).
\]
Absorbing the square root into the modulus $\xi$ finishes the proof.
\end{proof}

\begin{theorem}\label{thm:matrix-stability-reduction}
There exists a function $\omega\colon [0,1]\to [0,\infty)$ with $\omega(\varepsilon)\to 0$ as $\varepsilon\to 0$ such that the following holds.
For every $n\in\bbN$, every tracial von Neumann algebra $(M,\tau_M)$, and every map
$
\varphi\colon {\rm M}_{n}(\mathbb C)_{\le 1}\to M_{\le 1}
$
which is an $\varepsilon$-unital $L^2$-$\varepsilon$-$*$-homomorphism, there exist
\begin{itemize}[leftmargin=2em]
\item a semifinite tracial von Neumann algebra $(P,\Tr)$;
\item a trace-one projection $p\in P$ with $pPp\cong M$;
\item a finite projection $q\in P$;    
\item and a unital $*$-homomorphism
$
\psi\colon {\rm M}_{n}(\mathbb C)\to qPq
$
\end{itemize}
such that, after identifying $M$ with $pPp$, one has
\[
\sup_{\|x\|\le 1}
\|\varphi(x)-\psi(x)\|_{2,\Tr}
\le \omega(\varepsilon).
\]
\end{theorem}

\begin{proof}
By Corollary~\ref{C.matrix-unitary-correction}, there exist a semifinite tracial von Neumann algebra $(P,\Tr)$, projections $p,e\in P$ with
\[
\Tr(p)=1,
\qquad
\Tr(e)\in [1-\alpha_1(\varepsilon),1+\alpha_1(\varepsilon)],
\qquad
pPp\cong M,
\]
and a continuous homomorphism
$
\pi\colon {\rm U}(n)\to \mathcal U(ePe)
$
\[
\sup_{u\in {\rm U}(n)}\|\varphi(u)-\pi(u)\|_{2,\Tr}\le \alpha_1(\varepsilon)
\]
for some modulus $\alpha_1(\varepsilon)\to 0$. After enlarging $\alpha_1$ if necessary, Corollary~\ref{C.matrix-unitary-correction} also gives
\[
\sup_{u\in {\rm U}(n)}
\left|
\Tr\bigl(\pi(u)\bigr)-\tau_n(u)
\right|
\le \alpha_1(\varepsilon).
\]
Together with the scalar-circle estimate,
\[
\sup_{z\in\mathbb T}\|\pi(zI_n)-ze\|_{2,\Tr}\le \alpha_1(\varepsilon).
\]
Applying Theorem~\ref{T.standard-representation-extraction}, we obtain a finite projection $q\le e$. There is also a $*$-homomorphism
$
\psi\colon {\rm M}_{n}(\mathbb C)\to qPq
$
with $\psi(1)=q$ such that
\[
\sup_{u\in {\rm U}(n)}\|\pi(u)-\psi(u)\|_{2,\Tr}\le \xi(\alpha_1(\varepsilon)).
\]
Put
$
\delta(\varepsilon):=\alpha_1(\varepsilon)+\xi(\alpha_1(\varepsilon)).
$
Then for every $u\in {\rm U}(n)$,
\[
\|\varphi(u)-\psi(u)\|_{2,\Tr}
\le
\|\varphi(u)-\pi(u)\|_{2,\Tr}
+
\|\pi(u)-\psi(u)\|_{2,\Tr}
\le \delta(\varepsilon).
\]

It remains to pass from unitary elements to the whole unit ball. Let $x\in {\rm M}_{n}(\mathbb C)$ with $\|x\|\le 1$, and write $x=a+ib$ with $a=a^*$ and $b=b^*$. Set
\[
u_a:=a+i(1-a^2)^{1/2},
\qquad
u_b:=b+i(1-b^2)^{1/2}.
\]
Then $u_a,u_b\in {\rm U}(n)$ and
$
 a=\frac12\bigl(u_a+u_a^*\bigr),$ and $b=\frac12\bigl(u_b+u_b^*\bigr),$
so
\[
 x=\frac12u_a+\frac12u_a^*+\frac{i}{2}u_b+\frac{i}{2}u_b^*.
\]
Since $\psi$ is linear and $*$-preserving, while $\varphi$ is approximately additive, approximately homogeneous, and approximately $*$-preserving on contractions, we obtain
\[
\begin{aligned}
\Bigl\|\varphi(x)
-\frac12\varphi(u_a)-\frac12\varphi(u_a^*)
-\frac{i}{2}\varphi(u_b)-\frac{i}{2}\varphi(u_b^*)\Bigr\|_{2,\Tr} \le 7\varepsilon.
\end{aligned}
\]
Therefore
\[
\begin{aligned}
\|\psi(x)-\varphi(x)\|_{2,\Tr}
&\le \frac12\|\psi(u_a)-\varphi(u_a)\|_{2,\Tr}
   +\frac12\|\psi(u_a^*)-\varphi(u_a^*)\|_{2,\Tr} \\
&\qquad +\frac12\|\psi(u_b)-\varphi(u_b)\|_{2,\Tr}
   +\frac12\|\psi(u_b^*)-\varphi(u_b^*)\|_{2,\Tr}
   +7\varepsilon \\
&\le 2\delta(\varepsilon)+7\varepsilon.
\end{aligned}
\]
Thus the assertion holds with
$
\omega(\varepsilon):=2\delta(\varepsilon)+7\varepsilon.
$
This proves the theorem.
\end{proof}

\section{Ulam stability for the hyperfinite II$_1$-factor}
\label{sec:hyperfinite}
In this section we prove one of our main results, Theorem~\ref{T.amplified.stability}, which says that the hyperfinite II$_1$ factor is Ulam stable. The main step is to use Theorem~\ref{thm:matrix-stability-reduction} to obtain a sequence of approximate matrix embeddings, and then to take a limit.

Fix an increasing sequence of full matrix subfactors
$
   A_1\subset A_2\subset\cdots\subset \R
$
such that
\[
   \overline{\bigcup_{n\geq1}A_n}^{\|\cdot\|_2}=\R.
\]
Let
$
   E_n:\R\to A_n
$
be the trace-preserving conditional expectation.

\begin{lemma}
\label{L.ucp.limit}
There exists a function
$\omega_0(\varepsilon)\to0$ for $\varepsilon\to0$ such that the following holds.  Let $M$ be a II$_1$-factor and let
$
   \varphi:\R_{\leq1}\to M_{\leq1}
$
be an $L^2$-$\varepsilon$-$*$-homomorphism and $\varepsilon$-unital.  Then there exists a normal unital trace-preserving completely positive map
$
   \psi:\R\to M
$
such that
\[
   \sup_{x\in\R_{\leq1}} \|\varphi(x)-\psi(x)\|_2
   \leq \omega_0(\varepsilon).
\]
Moreover, we have
\[
   \sup_{x,y\in\R_{\leq1}}
   \|\psi(xy)-\psi(x)\psi(y)\|_2
   \leq \omega_0(\varepsilon).
\]
\end{lemma}

\begin{proof}
Fix the standard semifinite factor
$
   P_0:=M\overline\otimes B(\ell^2)
$
with its canonical trace, and let
$
   p_0:=1\otimes e_{11}\in P_0.
$
We also set
\[
   P:=P_0\overline\otimes B(\ell^2),
   \qquad
   p:=p_0\otimes e_{11}\in P.
\]
Then $P$ is a semifinite factor and $pPp\cong M$ with $\Tr(p)=1$.

For every $n$, apply Theorem~\ref{thm:matrix-stability-reduction} to the restriction
\[
   \varphi|_{(A_n)_{\leq1}}:(A_n)_{\leq1}\to M_{\leq1}.
\]
Since every finite amplification corner of $M$ embeds, trace-preservingly, into $P_0$, we may realize the output inside $P_0$: there exist a finite projection
$
   q_n\in P_0
$
and a unital $*$-homomorphism
$
   \theta_n:A_n\to q_nP_0q_n
$
such that, after identifying $M$ with $p_0P_0p_0$, one has
\[
   \sup_{a\in (A_n)_{\leq1}}\|\varphi(a)-\theta_n(a)\|_{2,\Tr}
   \leq \omega(\varepsilon),
\]
where $\omega$ is the modulus from Theorem~\ref{thm:matrix-stability-reduction}.

Put
$
   \beta(\varepsilon):=\omega(\varepsilon)+\varepsilon.
$
Since $1\in A_n$ and $\varphi$ is $\varepsilon$-unital,
\[
   \|q_n-p_0\|_{2,\Tr}
   \leq
   \|q_n-\varphi(1)\|_{2,\Tr}+\|\varphi(1)-p_0\|_{2,\Tr}
   \leq \beta(\varepsilon).
\]
Hence, if
$
   t_n:=\Tr(q_n),
$
then
$
   |t_n-1|\leq \beta(\varepsilon)^2.
$

Let
$
   D_n:=\bigl(\theta_n(A_n)\otimes 1\bigr)'\cap (q_n\otimes 1)P(q_n\otimes 1).
$
Since $P$ is a semifinite factor and $\theta_n(A_n)$ is a full matrix algebra, $D_n$ is a semifinite factor. Moreover,
$
   q_n\otimes e_{11}\in D_n
$
has trace $t_n$. Choose a projection
$
   r_n\in D_n
$
with
$
   \Tr(r_n)=1
$
and
\[
   \|r_n-q_n\otimes e_{11}\|_{2,\Tr}^2=|1-t_n|.
\]
Indeed, if $t_n\geq 1$, choose $r_n\leq q_n\otimes e_{11}$ of trace $1$, and if $t_n\leq 1$, enlarge $q_n\otimes e_{11}$ by an orthogonal projection in $q_n\otimes (1-e_{11})$ of trace $1-t_n$.

Because $r_n$ commutes with $\theta_n(A_n)\otimes 1$, the map
\[
   \widetilde\theta_n:A_n\to r_nPr_n,
   \qquad
   \widetilde\theta_n(a):=(\theta_n(a)\otimes 1)r_n,
\]
is a unital $*$-homomorphism. For $a\in (A_n)_{\leq 1}$,
\[
\begin{aligned}
   \|\widetilde\theta_n(a)-\theta_n(a)\otimes e_{11}\|_{2,\Tr}
   &=\|(\theta_n(a)\otimes 1)\bigl(r_n-q_n\otimes e_{11}\bigr)\|_{2,\Tr} \\
   &\leq \|r_n-q_n\otimes e_{11}\|_{2,\Tr}
   \leq \beta(\varepsilon).
\end{aligned}
\]
Also,
\[
   \|r_n-p\|_{2,\Tr}
   \leq \|r_n-q_n\otimes e_{11}\|_{2,\Tr}+\|q_n-p_0\|_{2,\Tr}
   \leq 2\beta(\varepsilon).
\]

Since $\Tr(r_n)=\Tr(p)=1$, the projections $r_n$ and $p$ are Murray-von Neumann equivalent. By the standard polar-decomposition comparison for nearby equivalent projections, there exists a partial isometry
$
   v_n\in P
$
with
\[
   v_n^*v_n=r_n,
   \qquad
   v_nv_n^*=p,
\]
and
\[
   \|v_n-r_n\|_{2,\Tr}\leq 2\|p-r_n\|_{2,\Tr}
   \leq 4\beta(\varepsilon).
\]
Define
\[
   \widehat\theta_n:A_n\to pPp\cong M,
   \qquad
   \widehat\theta_n(a):=v_n\widetilde\theta_n(a)v_n^*.
\]
Then $\widehat\theta_n$ is a unital $*$-homomorphism. Since $\widetilde\theta_n(a)\in r_nPr_n$ and $\|\widetilde\theta_n(a)\|\leq 1$,
\[
\begin{aligned}
   \|\widehat\theta_n(a)-\widetilde\theta_n(a)\|_{2,\Tr}
   &\leq \|(v_n-r_n)\widetilde\theta_n(a)v_n^*\|_{2,\Tr}
      +\|r_n\widetilde\theta_n(a)(v_n^*-r_n)\|_{2,\Tr} \\
   &\leq 2\|v_n-r_n\|_{2,\Tr}
   \leq 8\beta(\varepsilon).
\end{aligned}
\]
Therefore, after identifying $\varphi(a)$ with $\varphi(a)\otimes e_{11}\in pPp$,
\[
\begin{aligned}
   \|\varphi(a)-\widehat\theta_n(a)\|_{2,\Tr}
   &\leq \|\varphi(a)-\theta_n(a)\otimes e_{11}\|_{2,\Tr}
      +\|\theta_n(a)\otimes e_{11}-\widetilde\theta_n(a)\|_{2,\Tr} \\
   &\qquad +\|\widetilde\theta_n(a)-\widehat\theta_n(a)\|_{2,\Tr} \\
   &\leq \omega(\varepsilon)+9\beta(\varepsilon).
\end{aligned}
\]
Set
$
   \alpha(\varepsilon):=\omega(\varepsilon)+9\beta(\varepsilon).
$
Then $\alpha(\varepsilon)\to 0$ as $\varepsilon\to 0$.

Define
$
   \psi_n:=\widehat\theta_n\circ E_n:\R\to M.
$
Then $\psi_n$ is normal, unital, trace-preserving and completely positive. Here $\widehat\theta_n$ is trace-preserving because it is a unital representation of a full matrix algebra into the II$_1$-factor $pPp\cong M$.

The space of unital completely positive maps $\R\to M$ is compact in the topology of pointwise ultraweak convergence.  Passing to a subnet, we may assume that $\psi_n$ converges pointwise ultraweakly to a unital completely positive map
$
   \psi:\R\to M.
$
Since $\tau_M\circ\psi_n=\tau_\R$ for every $n$, we get
\[
   \tau_M(\psi(x))=\tau_\R(x)
   \qquad (x\in\R).
\]
Thus $\psi$ is trace-preserving.

We record that $\psi$ is normal.  Let $0\leq x_i\uparrow x$ in $\R$.  Then
$
   \sup_i \psi(x_i)\leq \psi(x).
$
Let
$
   y:=\psi(x)-\sup_i\psi(x_i)\geq0.
$
Since $\psi$ is trace-preserving and $\tau_\R$ is normal,
\[
   \tau_M(y)=\tau_\R(x)-\sup_i\tau_\R(x_i)=0.
\]
The trace $\tau_M$ is faithful, hence $y=0$.  This proves normality.

Let $x\in\R_{\leq1}$.  Since $E_n(x)\in(A_n)_{\leq1}$ and $E_n(x)\to x$ in $\|\cdot\|_2$, Corollary~\ref{cor:quasi-isometry-lipschitz} gives
\[
   \|\varphi(x)-\varphi(E_n(x))\|_2
   \leq \|x-E_n(x)\|_2+\eta(\varepsilon).
\]
Hence
\[
\begin{aligned}
   \|\varphi(x)-\psi_n(x)\|_2
   &\leq \|\varphi(x)-\varphi(E_n(x))\|_2
      +\|\varphi(E_n(x))-\widehat\theta_n(E_n(x))\|_2  \\
   &\leq \|x-E_n(x)\|_2+\eta(\varepsilon)+\alpha(\varepsilon).
\end{aligned}
\]
Taking the limit inferior along the chosen subnet and using lower semicontinuity of the $2$-norm under ultraweak convergence gives
\[
   \|\varphi(x)-\psi(x)\|_2
   \leq \eta(\varepsilon)+\alpha(\varepsilon).
\]
The right-hand side is independent of $x$, and therefore
\[
   \sup_{x\in\R_{\leq1}}\|\varphi(x)-\psi(x)\|_2
   \leq \eta(\varepsilon)+\alpha(\varepsilon).
\]
Put
$
   \omega_1(\varepsilon):=\eta(\varepsilon)+\alpha(\varepsilon).
$

We now estimate the multiplicative defect of $\psi$.  Since $\psi$ is unital completely positive, it is contractive.  For $x,y\in\R_{\leq1}$,
\[
\begin{aligned}
   \|\psi(xy)-\psi(x)\psi(y)\|_2
   &\leq \|\psi(xy)-\varphi(xy)\|_2
        +\|\varphi(xy)-\varphi(x)\varphi(y)\|_2 \\
   &\quad +\|\varphi(x)\varphi(y)-\psi(x)\psi(y)\|_2.
\end{aligned}
\]
Since $\varphi$ takes values in $M_{\leq1}$ and $\psi$ is contractive, the last term is at most
\[
   \|\varphi(x)-\psi(x)\|_2+\|\varphi(y)-\psi(y)\|_2.
\]
Thus
$
   \|\psi(xy)-\psi(x)\psi(y)\|_2
   \leq \varepsilon+3\omega_1(\varepsilon).
$
The assertion follows with
$
   \omega_0(\varepsilon):=\max\{\omega_1(\varepsilon),\varepsilon+3\omega_1(\varepsilon)\}.
$
\end{proof}

We recall the Hilbert-module Stinespring construction in the tracial setting.  Let $\psi:\R\to M$ be a normal unital completely positive map.  Let $\mathcal H_\psi$ be the self-dual Hilbert right $M$-module obtained from $\R\odot M$ with respect to the $M$-valued inner product
\[
   \left\langle \sum_i x_i\otimes a_i,\sum_j y_j\otimes b_j\right\rangle_M
   =
   \sum_{i,j}a_i^*\psi(x_i^*y_j)b_j.
\]
Left multiplication on the first tensor factor gives a normal unital representation
$
   \pi_\psi:\R\to\mathcal L_M(\mathcal H_\psi),
$
and the adjointable map
\[
   V:M_M\to\mathcal H_\psi,
   \qquad
   V(a)=1\otimes a,
\]
is an isometry satisfying
$
   \psi(x)=V^*\pi_\psi(x)V.
$
Let
$
   p:=VV^*\in\mathcal L_M(\mathcal H_\psi).
$
We regard $M$ as the corner $p\mathcal L_M(\mathcal H_\psi)p$ via
$
   a\mapsto VaV^*.
$

The von Neumann algebra $\mathcal L_M(\mathcal H_\psi)$ is a semifinite amplification of $M$ and carries the canonical semifinite trace $\Tr_M$, normalized so that
$
   \Tr_M(p)=1.
$
We use the corresponding $2$-norm on finite-trace corners.

\begin{lemma}
\label{L.stinespring.projection}
Let $\psi:\R\to M$ be normal, unital and trace-preserving completely positive, and suppose that
\[
   \sup_{x,y\in\R_{\leq1}}
   \|\psi(xy)-\psi(x)\psi(y)\|_2\leq \delta.
\]
Let
$\psi(x)=V^*\pi(x)V$ and $p=VV^*$.
Then for every unitary $u\in\Unitary(\R)$,
$
   \|(1-p)\pi(u)p\|_{2,\Tr_M}^2
   \leq \delta
$
and
$
   \|p\pi(u)(1-p)\|_{2,\Tr_M}^2
   \leq \delta.
$
Consequently
$$
   \sup_{u\in\Unitary(\R)}\|[p,\pi(u)]\|_{2,\Tr_M}
   \leq 2\delta^{1/2}.
$$
\end{lemma}

\begin{proof}
For a unitary $u\in\Unitary(\R)$,
\[
\begin{aligned}
   \|(1-p)\pi(u)p\|_{2,\Tr_M}^2
   &=\Tr_M\bigl(p\pi(u)^*(1-p)\pi(u)p\bigr)  \\
   &=\Tr_M\bigl(p\pi(u^*u)p-p\pi(u)^*p\pi(u)p\bigr).
\end{aligned}
\]
Under the identification $p\mathcal L_M(\mathcal H_\psi)p\cong M$, this is
$
   \tau_M\bigl(\psi(1)-\psi(u)^*\psi(u)\bigr).
$
Since $\psi$ is unital,
$
   \psi(1)-\psi(u)^*\psi(u)
   =
   \psi(u^*u)-\psi(u)^*\psi(u).
$
The latter element is positive by Kadison's inequality and has $2$-norm at most $\delta$ by assumption.  Hence
\[
   0\leq \tau_M\bigl(\psi(1)-\psi(u)^*\psi(u)\bigr)
   \leq \|\psi(1)-\psi(u)^*\psi(u)\|_2
   \leq \delta.
\]
This proves the first estimate.  The second follows by applying the first to $u^*$.  Finally,
$
   [p,\pi(u)]=p\pi(u)(1-p)-(1-p)\pi(u)p,
$
and the two summands are orthogonal in $L^2(\Tr_M)$.  Thus the displayed estimate follows.
\end{proof}

The next lemma replaces an almost invariant projection by a nearby invariant projection.  The proof uses only Hilbert-space convexity.

\begin{lemma}
\label{L.average.commutant}
There exists a function
\[
   \kappa:[0,1]\to[0,\infty),
   \qquad
   \kappa(s)\to0 \quad (s\to0),
\]
with the following property.  Let $P$ be a semifinite von Neumann algebra with faithful normal semifinite trace $\Tr$, let
$
   \pi:\R\to P
$
be a normal unital representation, and let $p\in P$ be a projection with $\Tr(p)<\infty$.  Suppose that
\[
   \sup_{u\in\Unitary(\R)}\|[p,\pi(u)]\|_{2,\Tr}\leq s.
\]
Then there exists a projection
$
   q\in \pi(\R)'\cap P
$
such that
$
   \|p-q\|_{2,\Tr}\leq \kappa(s).
$
Moreover
$
   |\Tr(q)-\Tr(p)|\leq \kappa(s)^2.
$
\end{lemma}

\begin{proof}
Let $\alpha_u=\Ad(\pi(u))$.  Consider the closed convex hull, in the Hilbert space $L^2(P,\Tr)$, of the orbit
$
   \{\alpha_u(p):u\in\Unitary(\R)\}.
$
Call this closed convex set $C$.  Since
\[
   \|p-\alpha_u(p)\|_{2,\Tr}
   =
   \|[p,\pi(u)]\|_{2,\Tr}
   \leq s,
\]
the whole set $C$ is contained in the closed $L^2$-ball of radius $s$ around $p$.

Let $a\in C$ be the unique element of minimal $L^2$-norm.  Since $C$ is invariant under every $\alpha_u$ and each $\alpha_u$ acts isometrically on $L^2(P,\Tr)$, uniqueness of the minimal-norm element gives
\[
   \alpha_u(a)=a
   \qquad (u\in\Unitary(\R)).
\]
Moreover $a$ is an $L^2$-limit of convex combinations of positive contractions.  Hence $a$ is represented by a positive contraction in $P$.  Therefore
\[
   a\in \pi(\R)'\cap P,
   \qquad
   0\leq a\leq1,
\]
and
$
   \|p-a\|_{2,\Tr}\leq s.
$

We estimate $a-a^2$.  Since $p=p^2$ and $0\leq a\leq1$,
\[
\begin{aligned}
   \|a-a^2\|_{2,\Tr}
   &\leq
   \|a-p\|_{2,\Tr}
   +\|p-pa\|_{2,\Tr}
   +\|pa-a^2\|_{2,\Tr} \\
   &\leq 3\|a-p\|_{2,\Tr}
   \leq 3s .
\end{aligned}
\]
Let
$
   q:=1_{[1/2,1]}(a)\in\pi(\R)'\cap P.
$
Since $a\in L^2(P,\Tr)$ and $a\geq \frac12 q$, the projection $q$ has finite trace.  The scalar inequality
\[
   |t-1_{[1/2,1]}(t)|\leq 2|t-t^2|,
   \qquad 0\leq t\leq1,
\]
gives by functional calculus
$
   \|a-q\|_{2,\Tr}\leq2\|a-a^2\|_{2,\Tr}.
$
Hence
\[
   \|p-q\|_{2,\Tr}
   \leq
   \|p-a\|_{2,\Tr}+\|a-q\|_{2,\Tr}
   \leq 7s.
\]
Thus we may take $\kappa(s)=7s$ for $0\leq s\leq1$.

Finally, for finite projections $p$ and $q$,
\[
\begin{aligned}
   |\Tr(q)-\Tr(p)|
   &=
   \left|\Tr(q(1-p))-\Tr(p(1-q))\right|  \\
   &\leq
   \Tr(q(1-p))+\Tr(p(1-q))
   =
   \|p-q\|_{2,\Tr}^2.
\end{aligned}
\]
This proves the last assertion.
\end{proof}

The following elementary estimate will be used to compare the original compression with the nearby reducing compression.

\begin{lemma}
\label{L.change.projection}
Let $P$ be a semifinite von Neumann algebra with trace $\Tr$, let $p,q\in P$ be finite projections, and let $T\in P$ satisfy $\|T\|\leq1$.  Then
\[
   \|qTq-pTp\|_{2,\Tr}
   \leq 2\|p-q\|_{2,\Tr}.
\]
\end{lemma}

\begin{proof}
We write
$
   qTq-pTp=(q-p)Tq+pT(q-p).
$
Therefore
\[
   \|qTq-pTp\|_{2,\Tr}
   \leq \|(q-p)Tq\|_{2,\Tr}+\|pT(q-p)\|_{2,\Tr}
   \leq 2\|p-q\|_{2,\Tr}.
\]
\end{proof}

We are now ready to prove our main result.

\begin{theorem}
\label{T.amplified.stability}
There exist functions
$\Omega:[0,1]\to[0,\infty)$ and $\Gamma:[0,1]\to[0,\infty)$ with $\Omega(\varepsilon)\to0$ and $\Gamma(\varepsilon)\to0$ as $\varepsilon\to0$, such that the following holds.
Let $M$ be a II$_1$-factor and let
$
   \varphi:\R_{\leq1}\to M_{\leq1}
$
be an $L^2$-$\varepsilon$-$*$-homomorphism and $\varepsilon$-unital.  Then there exist
\begin{itemize}[leftmargin=2em]
\item a semifinite amplification $P$ of $M$;
\item a trace-one projection $p\in P$ with $pPp\cong M$;
\item a projection $q\in P$ with
$
   t:=\Tr_M(q)
   \in[1-\Gamma(\varepsilon),1+\Gamma(\varepsilon)];
$
\item and a unital normal $*$-homomorphism
$
   \Theta_t:\R\to qPq\cong M^t;
$
\end{itemize}
such that, after identifying $M$ with $pPp$, one has
\[
   \sup_{x\in\R_{\leq1}}
   \|\varphi(x)-\Theta_t(x)\|_{2,\Tr_M}
   \leq \Omega(\varepsilon).
\]
\end{theorem}

\begin{proof}
We first prove the assertion for all sufficiently small $\varepsilon$.  Apply Lemma~\ref{L.ucp.limit} to obtain a normal unital trace-preserving completely positive map
$
   \psi:\R\to M
$
satisfying
$
   \sup_{x\in\R_{\leq1}}\|\varphi(x)-\psi(x)\|_2
   \leq \omega_0(\varepsilon)
$
and
$
   \sup_{x,y\in\R_{\leq1}}
   \|\psi(xy)-\psi(x)\psi(y)\|_2
   \leq \omega_0(\varepsilon).
$
Let
$
   \psi(x)=V^*\pi(x)V
$
be the Hilbert-module Stinespring dilation described above.  Put
\[
   P:=\mathcal L_M(\mathcal H_\psi),
   \qquad
   p:=VV^*.
\]
Then $P$ is a semifinite amplification of $M$, $\Tr_M(p)=1$, and the corner $pPp$ is canonically identified with $M$.

By Lemma~\ref{L.stinespring.projection},
\[
   \sup_{u\in\Unitary(\R)}\|[p,\pi(u)]\|_{2,\Tr_M}
   \leq 2\omega_0(\varepsilon)^{1/2}.
\]
By Lemma~\ref{L.average.commutant}, there exists a projection
$
   q\in\pi(\R)'\cap P
$
such that
$
   \|p-q\|_{2,\Tr_M}
   \leq \kappa\bigl(2\omega_0(\varepsilon)^{1/2}\bigr).
$
Set
$
   \gamma(\varepsilon):=
   \kappa\bigl(2\omega_0(\varepsilon)^{1/2}\bigr).
$
Then
$
   |\Tr_M(q)-1|
   \leq \gamma(\varepsilon)^2.
$
For $\varepsilon$ sufficiently small, $\gamma(\varepsilon)^2<1$, so $t:=\Tr_M(q)>0$ and $qPq\cong M^t$.

Since $q$ commutes with $\pi(\R)$, the map
\[
   \Theta_t:\R\to qPq,
   \qquad
   \Theta_t(x):=q\pi(x)q,
\]
is a unital $*$-homomorphism into the corner $qPq$, whose unit is $q$.

We now compare $\Theta_t$ with $\varphi$.  Under the identification $M\cong pPp$, the element $\psi(x)$ corresponds to
$
   p\pi(x)p.
$
Hence, for $x\in\R_{\leq1}$,
\[
\begin{aligned}
   \|\varphi(x)-\Theta_t(x)\|_{2,\Tr_M}
   &\leq \|\varphi(x)-\psi(x)\|_2
      +\|p\pi(x)p-q\pi(x)q\|_{2,\Tr_M} \\
   &\leq \omega_0(\varepsilon)+2\|p-q\|_{2,\Tr_M} \\
   &\leq \omega_0(\varepsilon)+2\gamma(\varepsilon).
\end{aligned}
\]
Thus, for small $\varepsilon$, the theorem holds with
\[
   \Omega(\varepsilon):=\omega_0(\varepsilon)+2\gamma(\varepsilon),
   \qquad
   \Gamma(\varepsilon):=\gamma(\varepsilon)^2.
\]

It remains only to extend the statement to all $\varepsilon\in[0,1]$.  Choose $\varepsilon_0>0$ such that $\Gamma(\varepsilon)<1$ for $0<\varepsilon\leq\varepsilon_0$.  For $\varepsilon>\varepsilon_0$, enlarge $\Omega$ and $\Gamma$ on $[\varepsilon_0,1]$ if necessary.  Indeed, every II$_1$-factor contains a unital copy of $\R$.  Taking
\[
   P=M\otimes B(\ell^2),
   \qquad
   p=1\otimes e_{11},
   \qquad
   q=1\otimes e_{22},
\]
and any unital embedding $\R\hookrightarrow qPq\cong M$, one obtains the required conclusion with a uniform coarse bound.  This redefinition away from $0$ does not affect the limits
$\Omega(\varepsilon)\to0,$ and $\Gamma(\varepsilon)\to0$ as $\varepsilon\to0$.
\end{proof}

\begin{remark}
Tracking the moduli in the proof gives the following rough quantitative
bounds:
\[
\Gamma(\varepsilon)=O(\varepsilon^{1/24}),
\qquad
\Omega(\varepsilon)=O(\varepsilon^{1/48}).
\]
For this we need to combine the quantitative results of \cite{dCOT} with the estimates of this paper. No attempt is made here to optimize these exponents.
\end{remark}

\section{The hyperfinite II$_1$-factor is isolated}
\label{sec:applications}

In this section we record a consequence of Theorem~\ref{T.amplified.stability}
which shows that a sufficiently accurate approximately $*$-isomorphic copy of the hyperfinite
factor must again be hyperfinite. This can be interpreted as a form of isolation of the hyperfinite II$_1$-factor in a suitable space of II$_1$ factors.

\begin{definition}
Let \(M\) and \(N\) be tracial von Neumann algebras. A map
\(\varphi\colon M_{\leq 1}\to N_{\leq 1}\) is called \(\delta\)-surjective if
\[
   \sup_{y\in N_{\leq 1}}
   \inf_{x\in M_{\leq 1}}
   \|\varphi(x)-y\|_2
   \leq \delta .
\]
We say that $\varphi \colon M_{\leq 1}\to N_{\leq 1}$ is an $\delta$-$*$-isomorphism if it is a $\delta$-surjective, $\delta$-unital, $L^2$-$\delta$-$*$-homomorphism.
\end{definition}

We shall use the following elementary consequence of Theorem~\ref{T.amplified.stability}.

\begin{lemma}
\label{L.almost-surjective-gives-dense-homomorphic-image}
Let \(M\) be a \(\mathrm{II}_1\)-factor and let
$
   \varphi\colon \R_{\leq 1}\to M_{\leq 1}
$
be an \(\varepsilon\)-unital $L^2$-\(\varepsilon\)-\(*\)-homomorphism which is
\(\varepsilon\)-surjective. Assume that \(\Gamma(\varepsilon)<1\). Let
\[
   \Theta_t\colon \R\to qPq\cong M^t
\]
be the homomorphism obtained from Theorem~\ref{T.amplified.stability}, where
\(P\) is a semifinite amplification of \(M\), \(pPp\cong M\), \(\Tr(p)=1\),
\(\Theta_t(1)=q\), and \(t=\Tr(q)\in[1-\Gamma(\varepsilon),1+\Gamma(\varepsilon)]\).
Put
\[
   Q:=\Theta_t(\R)\subset qPq .
\]
Then \(Q_{\leq 1}\) is \(\alpha(\varepsilon)\)-dense in \((qPq)_{\leq 1}\), with
respect to the normalized trace on \(qPq\), where
\[
   \alpha(\varepsilon)
   :=
   \frac{3(\varepsilon+\Omega(\varepsilon))}
        {(1-\Gamma(\varepsilon))^{1/2}} .
\]
\end{lemma}

\begin{proof}
We identify \(M\) with \(pPp\). Since \(\varphi\) is \(\varepsilon\)-unital and
\(\Theta_t(1)=q\), Theorem~\ref{T.amplified.stability} gives
\[
   \|p-q\|_{2,\Tr}
   \leq
   \|p-\varphi(1)\|_{2,\Tr}
   +
   \|\varphi(1)-\Theta_t(1)\|_{2,\Tr}
   \leq
   \varepsilon+\Omega(\varepsilon).
\]
Let \(y\in(qPq)_{\leq 1}\). Then \(pyp\in(pPp)_{\leq 1}\). By
\(\varepsilon\)-surjectivity of \(\varphi\), there exists \(x\in\R_{\leq 1}\)
such that
\[
   \|pyp-\varphi(x)\|_{2,\Tr}\leq\varepsilon .
\]
Using Lemma~\ref{L.change.projection}, we obtain
\[
\begin{aligned}
   \|y-\Theta_t(x)\|_{2,\Tr}
   &\leq
   \|qyq-pyp\|_{2,\Tr}
   +
   \|pyp-\varphi(x)\|_{2,\Tr}
   +
   \|\varphi(x)-\Theta_t(x)\|_{2,\Tr}       \\
   &\leq
   2\|p-q\|_{2,\Tr}+\varepsilon+\Omega(\varepsilon) \\
   &\leq
   3(\varepsilon+\Omega(\varepsilon)).
\end{aligned}
\]
Passing to the normalized trace on \(qPq\) divides the \(2\)-norm by
\(t^{1/2}\). Since \(t\geq1-\Gamma(\varepsilon)\), the asserted density
estimate follows.
\end{proof}

We next isolate a simple consequence of \(L^2\)-density for the relative
commutant.

\begin{lemma}
\label{L.relative-commutant-large-scalar-corner}
Let \(Q\subset N\) be an inclusion of \(\mathrm{II}_1\)-factors. Assume that
\(Q_{\leq 1}\) is \(\alpha\)-dense in \(N_{\leq 1}\), with \(0<\alpha<1/4\).
Let
$
   A:=Q'\cap N .
$
Put
\[
   \eta(\alpha)
   :=
   \frac{1-\sqrt{1-4\alpha^2}}{2}.
\]
Then there exists a central projection \(z\in Z(A)\) such that
\[
   \tau_N(z)\geq 1-\eta(\alpha)
   \qquad\text{and}\qquad
   Az=\mathbb C z .
\]
Consequently,
$
   (Qz)'\cap zNz=\mathbb C z .
$
Moreover, \(Qz\subset zNz\) is
$
\alpha \tau_N(z)^{-1/2}
$
-dense with respect to the normalized trace on \(zNz\).
\end{lemma}

\begin{proof}
Let \(E_Q\colon N\to Q\) be the trace-preserving conditional expectation.
Since \(E_Q\) is the \(L^2\)-orthogonal projection onto \(L^2(Q)\), the
density assumption implies
\[
   \|x-E_Q(x)\|_2\leq\alpha
   \qquad (x\in N_{\leq 1}).
\]
If \(a\in A_{\leq 1}\), then \(E_Q(a)\in Q\cap Q'=\mathbb C1\). Hence every
contraction in \(A\) is within \(L^2\)-distance \(\alpha\) of a scalar.

In particular, if \(e\in A\) is a projection and \(s:=\tau_N(e)\), then
\(E_Q(e)=s1\), and therefore
\[
   s(1-s)
   =
   \|e-s1\|_2^2
   =
   \|e-E_Q(e)\|_2^2
   \leq
   \alpha^2 .
\]
Thus every projection \(e\in A\) satisfies
\[
   \tau_N(e)\leq \eta(\alpha)
   \quad\text{or}\quad
   \tau_N(e)\geq 1-\eta(\alpha).
\]

Applying this to the center \(Z(A)\), we get a central atom
\(z\in Z(A)\) with \(\tau_N(z)\geq1-\eta(\alpha)\). Otherwise one could
construct a central projection whose trace lies strictly between
\(\eta(\alpha)\) and \(1-\eta(\alpha)\), contradicting the preceding
projection gap.

It remains to show that \(Az=\mathbb C z\). Suppose not. Since \(Az\) is a
non-scalar finite factor, it contains a projection \(f\) whose normalized
trace in \(Az\) lies between \(1/3\) and \(2/3\). Hence
\[
   \frac{\tau_N(z)}{3}
   \leq
   \tau_N(f)
   \leq
   \frac{2\tau_N(z)}{3}.
\]
For \(\eta(\alpha)<1/4\), this gives
$
   \eta(\alpha)<\tau_N(f)<1-\eta(\alpha),
$
again contradicting the projection gap. Thus \(Az=\mathbb C z\).

Finally let \(y\in(zNz)_{\leq1}\). Choose \(x\in Q_{\leq1}\) with
\(\|y-x\|_{2,\tau_N}\leq\alpha\). Since \(z\in Q'\cap N\), one has
\(zx\in(Qz)_{\leq1}\), and
\[
   \|y-zx\|_{2,\tau_{zNz}}
   =
   \tau_N(z)^{-1/2}\|z(y-x)\|_{2,\tau_N}
   \leq
   \frac{\alpha}{\tau_N(z)^{1/2}} .
\]
This proves the density assertion.
\end{proof}

The following consequence of Popa's asymptotic orthogonalization theorem is
the only external subfactor-theoretic input needed below.

\begin{lemma}
\label{L.irreducible-dense-implies-finite-index}
Let \(Q\subset N\) be an irreducible inclusion of \(\mathrm{II}_1\)-factors.
Assume that \(Q_{\leq1}\) is \(\beta\)-dense in \(N_{\leq1}\), for some
\(\beta<1\). Then
$
   [N:Q]<\infty .
$
\end{lemma}

\begin{proof}
Assume, towards a contradiction, that \([N:Q]=\infty\). Since \(Q'\cap N=\mathbb C\),
the subfactor \(Q\subset N\) has infinite index under every non-zero projection
in \(Q'\cap N\). By Popa's asymptotic orthogonalization theorem
\cite{Popa2019}, for every separable von Neumann subalgebra \(B\subset N^\omega\)
there exists a unitary \(u\in N^\omega\) such that \(uBu^*\) is orthogonal to
\(Q^\omega\).

Choose a Haar unitary \(w\in N\), and put \(B:=W^*(w)\subset N\subset N^\omega\).
Then there exists a unitary \(u\in N^\omega\) such that
$
   uBu^*\perp Q^\omega .
$
Set
$
   v:=uwu^*\in N^\omega .
$
Then \(v\) is a unitary, \(\tau_\omega(v)=0\), and orthogonality gives
$
   E_{Q^\omega}(v)=0 .
$
Represent \(v\) by unitaries \(v_n\in N\). Since conditional expectations pass
to ultrapowers,
\[
   \lim_{n\to\omega}\|E_Q(v_n)\|_2=0 .
\]
But \(E_Q\) is the \(L^2\)-orthogonal projection onto \(L^2(Q)\), and
\(E_Q(v_n)\in Q_{\leq1}\). Therefore
\[
   \operatorname{dist}_2(v_n,Q_{\leq1})^2
   =
   \|v_n-E_Q(v_n)\|_2^2
   =
   1-\|E_Q(v_n)\|_2^2
   \longrightarrow_\omega 1 .
\]
This contradicts \(\beta\)-density for \(\beta<1\). Hence \([N:Q]<\infty\).
\end{proof}

We can now prove the desired rigidity statement.

\begin{theorem}
\label{T.almost-surjective-R-map-forces-hyperfinite}
There exists \(\delta_0>0\) such that: If \(M\) is a
\(\mathrm{II}_1\)-factor and there exists a \(\delta_0\)-\(*\)-isomorphism
$
   \varphi\colon \R_{\leq1}\to M_{\leq1},
$
then
$
   \R\cong M .
$
\end{theorem}

\begin{proof}
Choose \(\delta_0>0\) so small that
$
   \Gamma(\delta_0)<1,
$
and, with
\[
   \alpha_0
   :=
   \frac{3(\delta_0+\Omega(\delta_0))}
        {(1-\Gamma(\delta_0))^{1/2}},
\]
one has
\[
   \alpha_0<\frac14,\qquad
   \eta(\alpha_0)<\frac14,\qquad
   \frac{\alpha_0}{(1-\eta(\alpha_0))^{1/2}}<1 .
\]
This is possible since \(\Omega(\varepsilon)\to0\) and
\(\Gamma(\varepsilon)\to0\) as \(\varepsilon\to0\).

Apply Theorem~\ref{T.amplified.stability} to \(\varphi\). We obtain a
semifinite amplification \(P\) of \(M\), finite projections \(p,q\in P\) with
$pPp\cong M,$ and $qPq\cong M^t$,
where \(t=\Tr(q)\in[1-\Gamma(\delta_0),1+\Gamma(\delta_0)]\), and a unital
normal \(*\)-homomorphism
$
   \Theta_t\colon \R\to qPq .
$
Put
\[
   N:=qPq,\qquad Q:=\Theta_t(\R)\subset N .
\]
By Lemma~\ref{L.almost-surjective-gives-dense-homomorphic-image},
\(Q_{\leq1}\) is \(\alpha_0\)-dense in \(N_{\leq1}\).

Apply Lemma~\ref{L.relative-commutant-large-scalar-corner} to the inclusion
\(Q\subset N\). We obtain a projection
$
   z\in Z(Q'\cap N)
$
such that
\[
   \tau_N(z)\geq1-\eta(\alpha_0),
   \qquad
   (Qz)'\cap zNz=\mathbb C z .
\]
Moreover, the irreducible inclusion
$
   Qz\subset zNz
$
is $\delta$-dense with
\[
   \delta=\frac{\alpha_0}{\tau_N(z)^{1/2}}
   \leq
   \frac{\alpha_0}{(1-\eta(\alpha_0))^{1/2}}
   <1.
\]

By Lemma~\ref{L.irreducible-dense-implies-finite-index}, the inclusion
$
   Qz\subset zNz
$
has finite Jones index. Since \(Qz\cong Q\cong \R\), the subfactor \(Qz\) is
hyperfinite. Finite-index extensions of hyperfinite \(\mathrm{II}_1\)-factors
are hyperfinite; equivalently, amenability is preserved under finite-index
extensions of subfactors \cite{PimsnerPopa1986,Popa1994}. Hence \(zNz\) is
hyperfinite.

Since \(N\) is a \(\mathrm{II}_1\)-factor and \(z\neq0\), the algebra \(N\) is
an amplification of the corner \(zNz\). Therefore \(N\) is hyperfinite. But
\(N\cong M^t\), and hyperfiniteness is invariant under amplification. Thus
\(M\) is hyperfinite, and therefore isomorphic to \(\R\) by \cite{Connes1976}.
\end{proof}

\section{Open problems}
\label{sec:open-problems}

We end this article with some open problems that we find interesting.

\begin{openproblem} Let $M$ be a non-injective II$_1$-factor or for that matter $M=L(\Gamma)$ for a non-amenable i.c.c.\ group or even a non-abelian free group. Does $(M,\tau)$ satisfy the conclusion of Theorem~\ref{T.amplified.stability}?
\end{openproblem} 

In the operator-norm setting, one important source of approximate
homomorphisms comes from Kadison--Kastler perturbation theory.  If two
$C^*$-subalgebras of a common ambient algebra have operator-norm unit balls
which are close in Hausdorff distance, then choosing, for each element of one
unit ball, a nearby element of the other unit ball gives a nonlinear
approximate $*$-homomorphism.  In this way perturbation questions for
subalgebras naturally produce examples of the approximate maps studied in
operator-norm Ulam stability.

The analogous source of examples is much less significant in the tracial
setting.  Let \(A,B\subset (M,\tau)\) be unital von Neumann subalgebras and
let \(E_A,E_B\) be the trace-preserving conditional expectations.  If
\[
d_{H,2}(A_{\leq 1},B_{\leq 1})\leq \delta,
\]
where the Hausdorff distance is computed with respect to \(\|\cdot\|_2\) on
the operator-norm unit balls, then equivalently
\[
\sup_{a\in A_{\leq 1}}\|a-E_B(a)\|_2\leq \delta
\quad\text{and}\quad
\sup_{b\in B_{\leq 1}}\|b-E_A(b)\|_2\leq \delta .
\]
Indeed, \(E_B(a)\) is the \(L^2\)-orthogonal projection of \(a\) onto
\(L^2(B)\), and \(E_B\) is operator-norm contractive.

Thus there is a canonical comparison map
$
E_B|_{A_{\leq 1}}\colon A_{\leq 1}\to B_{\leq 1}.
$
It is unital, trace-preserving, completely positive, linear and
\(*\)-preserving.  Moreover, for \(a,b\in A_{\leq 1}\),
\[
\begin{aligned}
\|E_B(ab)-E_B(a)E_B(b)\|_2
&\leq \|E_B(ab)-ab\|_2
   +\|ab-E_B(a)E_B(b)\|_2       \\
&\leq \delta
   +\|(a-E_B(a))b\|_2
   +\|E_B(a)(b-E_B(b))\|_2      \\
&\leq 3\delta .
\end{aligned}
\]
Hence Hausdorff closeness in trace norm of subalgebras automatically produces
a unital trace-preserving completely positive almost multiplicative map.
For such maps the Stinespring construction already gives the natural
comparison with a representation into an amplification corner.  Consequently,
unlike in the operator-norm Kadison--Kastler setting, this construction does
not provide a particularly interesting source of $L^2$-approximate
\(*\)-homomorphisms.

Thus, we finish with the following question.

\begin{openproblem}
Find a natural source of examples of non-trivial $L^2$-approximate \(*\)-homomorphisms or prove that none exist.
\end{openproblem}

\section*{Acknowledgements}

The authors are grateful to Ben De Bondt and Ilijas Farah for helpful discussions on closely related topics.

\medskip

GPT-5.5 was used to assist in drafting parts of this manuscript. All content was reviewed and substantially revised by the authors, who are responsible for the final text.

\end{document}